\documentclass[a4paper,fleqn]{cas-sc}

\usepackage[numbers]{natbib}

\def\tsc#1{\csdef{#1}{\textsc{\lowercase{#1}}\xspace}}
\tsc{WGM}
\tsc{QE}
\usepackage{tikz}
\usetikzlibrary{arrows.meta}
\usepackage{placeins}
\usepackage{xcolor}

\newtheorem{theorem}{Theorem}
\newtheorem{lemma}{Lemma}
\newdefinition{rmk}{Remark}
\newdefinition{definition}{Definition}
\newdefinition{proposition}{Proposition}
\newdefinition{corollary}{Corollary}
\newdefinition{example}{Example}
\newproof{pf}{Proof}
\newproof{pf-full}{Proof in the full-dimensional case}
\newproof{pot1}{Proof of Theorem~\ref{thm:subsets-of-unit-hypercube}}
\newproof{pot2}{Proof of Theorem~\ref{thm:functions-on-unit-hypercube}}
\newproof{pot3}{Proof of Theorem~\ref{thm:local-l-extendability-on-lattice}}

\newtheorem{appthm}{Theorem}[section]
\newtheorem{applemma}[appthm]{Lemma}

\newproof{polro}{Proof of Lemma~\ref{lem:halfUJ_triangulable_implies_relative_opposite_pair}}
\newproof{pol:lem:argmin_of_local_convex_extension}{Proof of Lemma~\ref{lem:argmin_of_local_convex_extension}}
\newproof{pol:lem:intersection_of_argmin_convex_hulls}{Proof of Lemma~\ref{lem:intersection_of_argmin_convex_hulls}}
\newproof{pol:lem:cube-d-convex-affinity-domain-characterization}{Proof of Lemma~\ref{lem:cube-d-convex-affinity-domain-characterization}}
\newproof{pol:lem:local-extension-commutes-with-cube-restriction}{Proof of Lemma~\ref{lem:local-extension-commutes-with-cube-restriction}}

\newcommand{\argmin}{\mathop{\rm argmin}\limits}

\newcommand{\dom}{\operatorname{dom}}

\usepackage[mathlines]{lineno}

\ExplSyntaxOn
\cs_set:Npn \__first_footerline:
  { \group_begin: \small \sffamily \__short_authors: \group_end: }
\ExplSyntaxOff

\begin{document}
\let\WriteBookmarks\relax

\shorttitle{A Characterization of Functions on the Integer Lattice That Are L-Extendable through Local Convex Extensions}    

\shortauthors{Yokoyama et al.}  

\title [mode = title]{A Characterization of Functions on the Integer Lattice That Are L-Extendable through Local Convex Extensions}  



%

\author[1]{Ken Yokoyama}[orcid=0009-0005-6641-9394]

\cormark[1]


\ead{yokoyama@agent.inf.kyushu-u.ac.jp}



\affiliation[1]{organization={Kyushu University},
            addressline={744 Motooka, Nishi-ku}, 
            city={Fukuoka},
            postcode={819-0395}, 
            country={Japan}}

\author[2]{Yuni Iwamasa}




\author[1]{Kei Kimura}

\affiliation[2]{organization={Kobe University},
            addressline={1-1, Rokkodai-cho, Nada-ku}, 
            city={Kobe},
            postcode={657-8501}, 
            country={Japan}}

\cortext[1]{Corresponding author}



\begin{abstract}
$\mathrm{L}$-extendable functions on the integer lattice are functions that admit a \(\frac12\mathrm{UJ}\)-convex relaxation; that is, values can be assigned to the additional half-integral points, while keeping the original values at integer points, so that the resulting function is \(\frac12\mathrm{UJ}\)-convex.
Various examples of \(\mathrm{L}\)-extendable functions and their applications are known.
A general characterization of \(\mathrm{L}\)-extendability for functions on the integer lattice, however, remains open.

In this paper, we study the \(\mathrm{L}\)-extendability of functions on the integer lattice.
We first obtain a necessary condition: \(\mathrm{L}\)-extendable functions are integrally convex.
We then consider a special subclass of
\(\mathrm{L}\)-extendable functions, namely, those whose local convex
extensions give \(\mathrm{L}\)-convex relaxations.
This subclass has a useful optimization property: 
the optimal value of the original function coincides with that of the
\(\frac12\mathrm{UJ}\)-convex relaxation given by its local convex extension, a property not guaranteed for general
\(\frac12\mathrm{UJ}\)-convex relaxations.
We characterize this subclass by showing that a function is \(\mathrm{L}\)-extendable through its local convex extension if and only if it is integrally convex and its restriction to every unit hypercube is \(\mathrm{D}\)-convex, where \(\mathrm{D}\)-convexity is a discrete convexity associated with reflected Freudenthal simplicial divisions.
This subclass strictly contains the class of \(\mathrm{D}\)-convex functions.
We also give a combinatorial characterization of
\(\mathrm{D}\)-convex subsets of the vertices of a unit hypercube:
a set \(X\subseteq\{0,1\}^n\) is \(\mathrm{D}\)-convex if and only if
it is majority-closed and contains a relative opposite pair.
\end{abstract}



\begin{keywords}
Discrete convex analysis \sep
\(\mathrm{L}\)-extendability \sep
Integrally convex functions \sep
Local convex extension \sep
\(\mathrm{D}\)-convexity \sep
\(\mathrm{UJ}\)-convexity
\end{keywords}

\maketitle

\section{Introduction}\label{sec:introduction}

Discrete convex analysis extends concepts and methods of convex analysis to functions on discrete domains, providing a unified framework for nonlinear discrete optimization. Fundamental classes include submodular functions, \(\mathrm{L}^{\natural}\)-convex functions, \(\mathrm{M}^{\natural}\)-convex functions, and integrally convex functions; see, e.g., \cite{murota2003discrete}.  
Among these, integrally convex functions form a broad class that includes a number of standard classes of discrete convex functions.
They are characterized by the convexity of the \emph{local convex extension}, obtained by applying convexification locally on each integral unit hypercube \cite{FavatiTardella1990,murota2023recent}.
This definition gives a direct connection between functions on the integer lattice and ordinary convex functions on \(\mathbb R^n\).

Discrete convex analysis has also been extended from the integer lattice to more general graph structures.
Hirai~\cite{hirai2016discrete,hirai2018convexity} developed the theory of \(\mathrm{L}\)-convexity on oriented modular graphs, providing a framework that encompasses several discrete convexities, including \(k\)-submodularity \cite{huber2012towards}, skew-bisubmodularity \cite{huber2014skew}, tree-submodularity \cite{kolmogorov2011submodularity}, and \(\mathrm{UJ}\)-convexity \cite{fujishige2014bisubmodular}.

Within this graph-theoretic framework, a function is called \emph{\(\mathrm{L}\)-extendable} if it admits an \(\mathrm{L}\)-convex extension to an appropriate refinement of its underlying graph.
Such an extension is called an \(\mathrm{L}\)-convex relaxation and agrees with the original function at every original vertex.
$\mathrm{L}$-convex relaxations have applications to
Potts energy minimization~\cite{gridchyn2013potts},
fixed-parameter tractable algorithms~\cite{iwata2016half},
and multiflow and network design problems~\cite{hirai2015extendable};
see also~\cite{hirai2016k}.

For the integer lattice considered in this paper, we regard $\mathbb{Z}^n$ as the Cartesian product of $n$ copies of the path graph on $\mathbb{Z}$, with edges joining consecutive integers.
For this graph structure, $\mathrm{L}$-convex relaxations are precisely $\frac12\mathrm{UJ}$-convex relaxations on the half-integral lattice $(\frac12\mathbb{Z})^n$ ~\cite[Example~4, Theorem~4.5, and Example~12]{hirai2018convexity}.

Although various concrete examples of \(\mathrm{L}\)-extendable functions and their applications are known, a general structural understanding of \(\mathrm{L}\)-extendability is still lacking.
In particular, a complete characterization of functions on the integer lattice that admit an \(\mathrm{L}\)-convex relaxation remains open in arbitrary dimension.
In dimension two, however, \(\mathrm{L}\)-extendable functions are known to coincide with locally discrete midpoint convex functions~\cite[Remark~5]{moriguchi2020discrete}.
For functions on the vertices of a unit hypercube, a result of
Hirai and Iwamasa~\cite{hirai2016k} implies that a function on
\(\{0,1\}^n\) is \(\mathrm{L}\)-extendable if and only if its
effective domain is majority-closed.

Motivated by the lack of a general characterization,
our aim is to develop structural characterizations of
$\mathrm{L}$-extendable functions on the integer lattice.
As a first step toward this goal, we establish a general necessary condition: every \(\mathrm{L}\)-extendable function on the integer lattice is integrally convex.
This implication holds for any choice of an \(\mathrm{L}\)-convex relaxation.

Beyond the existence of an \(\mathrm{L}\)-convex relaxation, it is important for minimization whether the relaxation is exact.
A relaxation is called exact if it preserves the optimal value of the original problem.
The local convex extension always preserves this optimal
value, both on $\mathbb R^n$ and when restricted to the
half-integral lattice.
Indeed, every finite value of the extension is a convex
combination of values of the original function at neighboring
integer points, while the extension agrees with the original
function on $\mathbb Z^n$.
Consequently, whenever its restriction to the half-integral
lattice is $\frac12\mathrm{UJ}$-convex, it yields an exact
$\mathrm{L}$-convex relaxation.
Motivated by this optimization property, our goal is to characterize the functions on the integer lattice that are $\mathrm{L}$-extendable through their local convex extensions.

Our characterization involves \emph{$\mathrm{D}$-convex functions} introduced by Fujishige~\cite{fujishige2014bisubmodular}.
This class is based on discordant Freudenthal simplicial divisions, obtained by reflecting the Freudenthal division on each unit hypercube in a compatible way.
It includes $\mathrm{L}^{\natural}$-convex and $\mathrm{UJ}$-convex functions on the integer lattice ~\cite[Section~5.4]{fujishige2022discrete}.
A nonempty set of integer points is called  $\mathrm{D}$-convex if its indicator function, which takes the value zero on the set and positive infinity elsewhere, is  $\mathrm{D}$-convex.

We first show that the indicator function of a nonempty set $X\subseteq\{0,1\}^n$ is $\mathrm{L}$-extendable through its local convex extension if and only if $X$ is a $\mathrm{D}$-convex set.
We further characterize this condition combinatorially:
$X$ is $\mathrm{D}$-convex if and only if it is majority-closed
and contains a relative opposite pair.
This characterization can be compared directly with that of
Hirai and Iwamasa~\cite{hirai2016k}, which states that the indicator
function of $X$ admits an $\mathrm{L}$-convex relaxation if and only if
$X$ is majority-closed.

We then extend the equivalence with $\mathrm{D}$-convexity
from indicator functions to arbitrary functions
on the vertices of a unit hypercube.
That is, such a function is $\mathrm{L}$-extendable through
its local convex extension if and only if it is a
$\mathrm{D}$-convex function.

Finally, we establish our main characterization on the
integer lattice, extending the result for a unit hypercube.
Here, local $\mathrm{D}$-convexity is supplemented by
integral convexity as a global condition.
We prove that a function on the integer lattice is
$\mathrm{L}$-extendable through its local convex extension
if and only if it is integrally convex and its restriction
to every unit hypercube is $\mathrm{D}$-convex.
We also show that this class strictly contains
the class of $\mathrm{D}$-convex functions on the integer lattice.
Unlike $\mathrm{D}$-convexity on the integer lattice,
$\mathrm{D}$-convexity of the unit hypercube restrictions
does not require compatible choices of coordinate reflections.

\section{Preliminaries}
\label{sec:preliminaries}

In this section, we recall the basic notions of integral
convexity, \(\mathrm{L}\)-extendability, \(\frac12\mathrm{UJ}\)-convexity, and \(\mathrm{D}\)-convexity used throughout the paper.

\subsection{Notation}

Let $\mathbb Z$ and $\mathbb R$ denote the sets of integers and real numbers, respectively.
We write $\frac12\mathbb Z:=\{k/2\mid k\in\mathbb Z\}$ for the set of half-integers.
Throughout the paper, \(n\) is a positive integer, and functions are
allowed to take the value \(+\infty\).
For \(a\in\mathbb R\), we use the conventions
\(a<+\infty\) and
\(a+(+\infty)=(+\infty)+a=+\infty\).
We regard \(+\infty\le+\infty\) and
\(+\infty\ge+\infty\) as valid inequalities.
In convex combinations, a term with zero coefficient is understood not
to contribute, even when the corresponding function value is
\(+\infty\).

For a function
\(f:\mathbb Z^n\to\mathbb R\cup\{+\infty\}\),
its effective domain is
\(
\operatorname{dom} f
:=
\{x\in\mathbb Z^n\mid f(x)<+\infty\}.
\)
We allow \(\operatorname{dom} f=\emptyset\).
In this case, \(f(x)=+\infty\) for every \(x\in\mathbb Z^n\), and,
whenever an extension is considered, the function taking the value
\(+\infty\) at every point of the extended domain is admissible.
For a function \(f\) and a subset \(A\) of its domain, we denote the restriction of \(f\) to \(A\) by \(f|_A\).

For a vector \(x\) (in \(\mathbb{R}^n\) or \(\mathbb{Z}^n\)), we denote its \(i\)-th component by \(x_i\).
For \(x=(x_1,\ldots,x_n)\in\mathbb R^n\), define
\(\lfloor x\rfloor:=(\lfloor x_1\rfloor,\ldots,\lfloor x_n\rfloor)\)
and
\(\lceil x\rceil:=(\lceil x_1\rceil,\ldots,\lceil x_n\rceil)\).
We write \([n]:=\{1,\ldots,n\}\), and, for each \(i\in[n]\), let
\(e_i\in\mathbb Z^n\) denote the \(i\)th unit vector.
For \(A\subseteq[n]\), let
\(\chi_A\in\{0,1\}^n\) denote the characteristic vector of \(A\).
We denote the all-zero and all-one vectors in \(\mathbb Z^n\) by
\(
\mathbf 0:=(0,\ldots,0),~
\mathbf 1:=(1,\ldots,1),
\)
respectively.
For \(p,x\in\mathbb R^n\), let
\(\langle p,x\rangle:=\sum_{i=1}^n p_i x_i\).
For \(p\in\mathbb R^n\), the linear perturbation of \(f\) by \(p\)
is defined by
\(f[-p](x):=f(x)-\langle p,x\rangle\) for \(x\in\mathbb Z^n\).

For a set $S\subseteq\mathbb R^n$, we denote its convex hull by $\overline{S}$.
The closed convex hull is denoted by $\operatorname{cl}(\overline{S})$.

\subsection{Integral convexity and local convex extensions}

We recall the basic notions of integral convexity used in this paper.
For further background on integrally convex sets and functions, see~\cite{murota2003discrete,FavatiTardella1990,murota2023recent}.

For \(x\in\mathbb R^n\), the \emph{integral neighborhood} of \(x\) is
defined by
\begin{equation}
N(x)
:=
\{z\in\mathbb Z^n\mid |x_i-z_i|<1 \ (i=1,\ldots,n)\}.
\label{eq:integral-neighborhood}
\end{equation}
Equivalently,
\(
N(x)
=
\{z\in\mathbb Z^n\mid
\lfloor x_i\rfloor\le z_i\le\lceil x_i\rceil
\ (i=1,\ldots,n)\}.
\)

For a set \(S\subseteq\mathbb Z^n\) and \(x\in\mathbb R^n\), the convex
hull \(\overline{S\cap N(x)}\) is called the \emph{local convex hull} of
\(S\) around \(x\).  A nonempty set \(S\subseteq\mathbb Z^n\) is called
\emph{integrally convex} if the union of its local convex hulls is convex,
or equivalently,
\begin{equation}
\overline{S}
=
\bigcup_{x\in\mathbb R^n}
\overline{S\cap N(x)}.
\label{eq:integrally-convex-set}
\end{equation}

For a function
\(f:\mathbb Z^n\to\mathbb R\cup\{+\infty\}\), its
\emph{local convex extension}
\(\tilde f:\mathbb R^n\to\mathbb R\cup\{+\infty\}\) is defined by
\begin{equation}
\tilde f(x)
:=
\min
\left\{
\sum_{y\in N(x)}\lambda_y f(y)
\ \middle|\
\sum_{y\in N(x)}\lambda_y y=x,\ 
(\lambda_y)\in\Lambda(x)
\right\}
\qquad
(x\in\mathbb R^n),
\label{eq:local-convex-extension}
\end{equation}
where
\begin{equation}
\Lambda(x)
:=
\left\{
(\lambda_y\mid y\in N(x))
\ \middle|\
\sum_{y\in N(x)}\lambda_y=1,\ 
\lambda_y\ge0 \ (y\in N(x))
\right\}.
\label{eq:lambda-local-convex-extension}
\end{equation}
Thus \(\tilde f(x)\) is the value at \(x\) of the convexification of
the restriction of \(f\) to \(N(x)\);
see \cite{FavatiTardella1990} and
\cite[Sections~3.1--3.2]{murota2023recent}.
A function
\(f:\mathbb Z^n\to\mathbb R\cup\{+\infty\}\) is called
\emph{integrally convex} if its local convex extension \(\tilde f\) is
convex on \(\mathbb R^n\).  Since \(N(x)=\{x\}\) for
\(x\in\mathbb Z^n\), we have
\(\tilde f(x)=f(x)\) for all \(x\in\mathbb Z^n\).

\subsection{Discrete convexity based on simplicial divisions}
\label{subsec:discrete-convexity-simplicial-divisions}

The \(\mathrm{L}^{\natural}\)-convex, \(\mathrm{UJ}\)-convex, and
\(\mathrm{D}\)-convex functions considered in this subsection can all
be described through convex extensions associated with simplicial
divisions.
A common geometric foundation for these classes is the \emph{Freudenthal simplicial division} of the integer grid. 
We first recall this division.

For \(z\in\mathbb Z^n\) and a permutation \(\pi\) of \([n]\), let
\[
\Delta_{z,\pi}
:=
\overline{
\left\{
z,\,
z+e_{\pi(1)},\,
z+e_{\pi(1)}+e_{\pi(2)},\,
\ldots,\,
z+e_{\pi(1)}+e_{\pi(2)}+\cdots+e_{\pi(n)}
\right\}
}.
\]
Here, the permutation \(\pi\) specifies the order in which the
coordinate directions are added along the chain of vertices.
The Freudenthal simplicial division of \(z+[0,1]^n\) consists of the \(n\)-dimensional simplices \(\Delta_{z,\pi}\), where \(\pi\) ranges over all permutations of \([n]\), together with all of their faces~\cite{freudenthal1942simplicial}.
For \(\mathrm{L}^{\natural}\)-convex functions, convex extension to
\(\mathbb R^n\) is obtained by affine interpolation over this
simplicial division; see \cite[Section~7.7]{murota2003discrete} and \cite[Section~16.3]{fujishige2005submodular}.

More generally, local convex extensions admit a simplicial description in which each point is represented using vertices from its integral neighborhood.
By \cite[Lemma~1]{IimuraMurotaTamura2005}, for every finite integrally convex set \(X\subseteq\mathbb Z^n\), there exists a simplicial division \(\mathcal T\) of \(\overline X\) such that, for every \(y\in\overline X\), all vertices of the smallest simplex \(T(y)\in\mathcal T\) containing \(y\) belong to \(N(y)\).
For a broader discussion of simplicial divisions in the context of integral convexity, see~\cite[Section~3.5]{murota2023recent}.

\subsubsection{\(\mathrm{L}\)-extendability and
\(\frac12\mathrm{UJ}\)-convexity}
\label{subsec:l-extendability-half-uj-convexity}

The notion of \(\mathrm{UJ}\)-convexity is defined analogously by replacing the Freudenthal division with the Union-Jack division \cite{todd1977union} and extending functions by affine interpolation over this simplicial division \cite{fujishige2014bisubmodular}.
In this paper, we use the half-integral rescaling of the Union-Jack simplicial division, which we call the \emph{\(\frac12\mathrm{UJ}\) simplicial division}.
More precisely, on the fundamental cube it is obtained from the Union-Jack simplicial division of \([-1,1]^n\) by the affine map
\(
u\longmapsto \frac12\mathbf 1+\frac12u
\)
onto \([0,1]^n\)
\cite[Section~4]{fujishige2014bisubmodular}.
Since the resulting division is invariant under translations by integer vectors, it suffices to describe its restriction to \([0,1]^n\).
We next describe the $\frac12\mathrm{UJ}$ simplicial division and introduce the corresponding discrete midpoint inequality.

Let \(c:=\frac12\mathbf 1\).
For a choice of signed unit vectors
\(s_i\in\{e_i,-e_i\}\) \((i\in[n])\) and a permutation
\(\pi\) of \([n]\), the corresponding \(n\)-dimensional
\(\frac12\mathrm{UJ}\)-simplex in \([0,1]^n\) is
\begin{equation}
\overline{
\left\{
c,\ 
c+\frac12s_{\pi(1)},\ 
c+\frac12\bigl(s_{\pi(1)}+s_{\pi(2)}\bigr),\ 
\ldots,\ 
c+\frac12\bigl(s_{\pi(1)}+s_{\pi(2)}+\cdots+s_{\pi(n)}\bigr)
\right\}
}.
\label{eq:half-uj-simplex-signed-chain}
\end{equation}
The restriction of the
\(\frac12\mathrm{UJ}\) simplicial division to \([0,1]^n\)
consists of these \(n\)-dimensional simplices together with all
of their faces.
We call every simplex in the \(\frac12\mathrm{UJ}\) simplicial division a \emph{\(\frac12\mathrm{UJ}\)-simplex}.
Every \(n\)-dimensional \(\frac12\mathrm{UJ}\)-simplex in
\([0,1]^n\) has \(c\) as a vertex and
\(c+\frac12\sum_{i=1}^n s_i\) as its unique \(0\)-\(1\) vertex.
Moreover, the simplex lies in the closed orthant determined by
this \(0\)-\(1\) vertex with respect to the center \(c\).
More precisely, if \(v\in\{0,1\}^n\) is the unique
\(0\)-\(1\) vertex of an \(n\)-dimensional
\(\frac12\mathrm{UJ}\)-simplex \(T\subseteq[0,1]^n\), then every
\(u\in T\) satisfies
\begin{equation}
v_i=1
\Longrightarrow
u_i\geq\frac12,
\qquad
v_i=0
\Longrightarrow
u_i\leq\frac12
\qquad
(i\in[n]).
\label{eq:half-uj-simplex-orthant-containment}
\end{equation}

For a function
\(g:(\frac12\mathbb Z)^n\to\mathbb R\cup{+\infty}\),
we define its \emph{\(\frac12\mathrm{UJ}\) simplicial extension}
\(\hat g:\mathbb R^n\to\mathbb R\cup{+\infty}\)
by affine interpolation over the simplices of the
\(\frac12\mathrm{UJ}\) simplicial division.
For \(x\in\mathbb R^n\), consider the unique simplex of the division whose relative interior contains \(x\).
If \(g\) is finite at every vertex of this simplex, define \(\hat g(x)\) as the value at \(x\) of the unique affine function on its affine hull that agrees with \(g\) at every vertex; otherwise, set \(\hat g(x):=+\infty\).

To define \(\frac12\mathrm{UJ}\)-convexity in terms of a discrete midpoint condition, we introduce the following two \(\frac12\mathrm{UJ}\)-rounding operations.
For \(u\in\frac14\mathbb Z\), define \([u]_{\frac12},[u]_1\in\frac12\mathbb Z\) by
\begin{equation}
[u]_{\frac12}
:=
\begin{cases}
u, & u\in\frac12\mathbb Z,\\
u+\frac14, & u-\frac14\in\mathbb Z,\\
u-\frac14, & u+\frac14\in\mathbb Z,
\end{cases}
\qquad
[u]_1
:=
\begin{cases}
u, & u\in\frac12\mathbb Z,\\
u-\frac14, & u-\frac14\in\mathbb Z,\\
u+\frac14, & u+\frac14\in\mathbb Z.
\end{cases}
\label{eq:half-uj-roundings}
\end{equation}
Thus, both operations fix every half-integer. For
\(u\in\frac14\mathbb Z\setminus\frac12\mathbb Z\),
\([u]_{\frac12}\) is the unique nearest nonintegral half-integer
to \(u\), whereas \([u]_1\) is the unique nearest integer to \(u\).
For \(u\in(\frac14\mathbb Z)^n\), define
\([u]_{\frac12},[u]_1\in(\frac12\mathbb Z)^n\) componentwise by
\[
[u]_{\frac12}
:=
\left(
[u_1]_{\frac12},[u_2]_{\frac12},\dots,[u_n]_{\frac12}
\right),
\qquad
[u]_1
:=
\left(
[u_1]_1,[u_2]_1,\dots,[u_n]_1
\right).
\]

A function
\(g:(\frac12\mathbb Z)^n\to\mathbb R\cup\{+\infty\}\)
is called \emph{\(\frac12\mathrm{UJ}\)-convex} if
\begin{equation}
g(x)+g(y)
\geq
g\left(
\left[\frac{x+y}{2}\right]_{\frac12}
\right)
+
g\left(
\left[\frac{x+y}{2}\right]_{1}
\right)
\qquad
\left(
x,y\in\left(\frac12\mathbb Z\right)^n
\right).
\label{eq:half-uj-convexity}
\end{equation}

Hirai~\cite{hirai2018convexity} extended the classical theory of
\(\mathrm{L}^{\natural}\)-convex functions on the integer lattice
to \(\mathrm{L}\)-convex functions on oriented modular graphs.
We do not recall the general graph-theoretic definition, since
only the specialization to the half-integral lattice is needed
here; a similar specialization is used in
\cite[Section~6]{kimura2024neighborhood}.
In this specialization, \(\mathrm{L}\)-convexity is equivalent
both to the \(\frac12\mathrm{UJ}\)-convexity defined by
\eqref{eq:half-uj-convexity} and to the convexity of the corresponding \(\frac12\mathrm{UJ}\) simplicial extension.
The following proposition records these equivalences after
half-integral rescaling.

\begin{proposition}[{\cite[Theorem~4.5 and
Example~12]{hirai2018convexity}}]
\label{prop:half-uj-convex-extension-characterization}
Let
\(g:(\frac12\mathbb Z)^n\to\mathbb R\cup\{+\infty\}\),
and let \(\hat g:\mathbb R^n\to\mathbb R\{+\infty\}\) denote its
\(\frac12\mathrm{UJ}\) simplicial extension.
Then the following conditions are equivalent.
\begin{enumerate}
\item[\textup{(i)}]
The function \(g\) is \(\frac12\mathrm{UJ}\)-convex.


\item[\textup{(ii)}]
The function \(\hat g\) is convex on \(\mathbb R^n\).
\end{enumerate}
\end{proposition}

A function
\(f:\mathbb Z^n\to\mathbb R\cup\{+\infty\}\)
is called \emph{\(\mathrm{L}\)-extendable}
(or \emph{\(\frac12\mathrm{UJ}\)-extendable})
if there exists a \(\frac12\mathrm{UJ}\)-convex function
\(g:(\frac12\mathbb Z)^n\to\mathbb R\cup\{+\infty\}\)
such that \(g(x)=f(x)\) for all \(x\in\mathbb Z^n\).
Such a \(g\) is called an \emph{\(\mathrm{L}\)-convex (or \(\frac12\mathrm{UJ}\)-convex) relaxation} of \(f\).

To characterize \(\mathrm{L}\)-extendability on a unit
hypercube, we use the result of Hirai and Iwamasa
\cite[Theorem~1]{hirai2016k} on the existence of
\(k\)-submodular relaxations.
In preparation for its specialization to \(k=2\), we introduce
\(2\)-submodularity and its relation to
\(\frac12\mathrm{UJ}\)-convexity.

For \(a,b\in\{0,1,2\}\), define
\begin{equation}
a\sqcap b
:=
\begin{cases}
a, & a=b,\\
0, & a\ne b,
\end{cases}
\qquad
a\sqcup b
:=
\begin{cases}
a, & b=0 \text{ or } a=b,\\
b, & a=0 \text{ and } b\ne0,\\
0, & \text{otherwise}.
\end{cases}
\label{eq:two-submodular-operations}
\end{equation}
The operations act componentwise on
\(\{0,1,2\}^n\).
A function
\(h:\{0,1,2\}^n\to\mathbb R\cup\{+\infty\}\)
is called \emph{\(2\)-submodular} if
\begin{equation}
h(a)+h(b)
\geq
h(a\sqcap b)+h(a\sqcup b)
\qquad
(a,b\in\{0,1,2\}^n).
\label{eq:two-submodularity}
\end{equation}
It is well known that \(2\)-submodularity coincides with
bisubmodularity; see, e.g.,
\cite{fujishige2005submodular,huber2012towards}.
We identify the three labels with the three half-integral points as follows.
Let
\(\varphi:\{0,1,2\}\to\{0,\frac12,1\}\) be defined by
\(\varphi(0)=\frac12\), \(\varphi(1)=0\), and
\(\varphi(2)=1\), and apply \(\varphi\) componentwise.
A coordinatewise check gives
\begin{equation}
\varphi(a\sqcap b)
=
\left[
\frac{\varphi(a)+\varphi(b)}{2}
\right]_{\frac12},
\qquad
\varphi(a\sqcup b)
=
\left[
\frac{\varphi(a)+\varphi(b)}{2}
\right]_{1}
\label{eq:two-submodular-half-uj-correspondence}
\end{equation}
for all \(a,b\in\{0,1,2\}^n\).
Since \(\{0,\frac12,1\}^n\) is closed under the two
\(\frac12\mathrm{UJ}\)-rounding operations, we also call a
function
\(g:\{0,\frac12,1\}^n\to\mathbb R\cup\{+\infty\}\)
\(\frac12\mathrm{UJ}\)-convex if
\eqref{eq:half-uj-convexity} holds for all
\(x,y\in\{0,\frac12,1\}^n\).
With this convention, the pullback \(g\circ\varphi\) is
\(2\)-submodular (equivalently, bisubmodular) if and only if \(g\) is
\(\frac12\mathrm{UJ}\)-convex.
For a function
\(f:\{0,1\}^n\to\mathbb R\cup\{+\infty\}\), a
\(\frac12\mathrm{UJ}\)-convex function
\(g:\{0,\frac12,1\}^n\to\mathbb R\cup\{+\infty\}\)
satisfying
\(
g(x)=f(x)
~(x\in\{0,1\}^n)
\)
is called a \emph{\(2\)-submodular relaxation} of \(f\) under the identification \(\varphi\).
We say that \(f\) is \(\mathrm{L}\)-extendable on the unit hypercube if it admits such a relaxation.

To state the characterization of
Hirai and Iwamasa~\cite[Theorem~1]{hirai2016k},
we recall the notion of a majority-closed set.
For \(x,y,z\in\{0,1\}^n\), let
\(\operatorname{maj}(x,y,z)\in\{0,1\}^n\) denote the
componentwise majority vector; that is, for each \(i\in[n]\),
\(\operatorname{maj}(x,y,z)_i\) is the value appearing at least
twice among \(x_i,y_i,z_i\).
A set \(X\subseteq\{0,1\}^n\) is called
\emph{majority-closed} if
\[
\operatorname{maj}(x,y,z)\in X
\qquad
(x,y,z\in X).
\]

Hirai and Iwamasa~\cite{hirai2016k} gave the following characterization of \(\mathrm{L}\)-extendability on the unit hypercube.
\begin{lemma}[Hirai--Iwamasa~{\cite[Theorem~1]{hirai2016k}},
\(k=2\)]
\label{cor:hirai-iwamasa-majority-closure}
Let \(f:\{0,1\}^n\to\mathbb R\cup\{+\infty\}\).
Then \(f\) is \(\mathrm{L}\)-extendable on the unit hypercube, or equivalently, admits a \(2\)-submodular relaxation, if and only if \(\dom f\) is majority-closed.
\end{lemma}

\subsubsection{\(\mathrm{D}\)-convex functions and
\(\mathrm{D}\)-convex sets}

To define \(\mathrm{D}\)-convexity on the integer lattice, we first
describe the reflected Freudenthal structure underlying it.
For each \(z\in\mathbb Z^n\), choose a subset
\(T_z\subseteq[n]\).
The reflection of this division by \(T_z\) is obtained by
replacing each vertex \(z+\chi_X\), \(X\subseteq[n]\), by
\(
z+\chi_{(X\setminus T_z)\cup(T_z\setminus X)}.
\)

For these reflected divisions to fit together across unit
hypercubes, the choices of \(T_z\) must satisfy a compatibility
condition.
We say that the family
\((T_z)_{z\in\mathbb Z^n}\) is \emph{compatible} if, whenever two
unit hypercubes have a nonempty common face, the reflected
Freudenthal simplicial divisions induced on that face coincide.
For a compatible family, the reflected divisions of all unit
hypercubes together form a simplicial division of
\(\mathbb R^n\), called a
\emph{discordant Freudenthal simplicial division}.

The corresponding \(\mathrm{D}\)-convexity can be expressed in terms
of two integer midpoint operators associated with a compatible family~\cite[Section~5.4]{fujishige2022discrete}.
\footnote{The notion called
\(\mathrm{D}\)-convexity here is termed
\(\mathrm{BS}^{*}\)-convexity in \cite[Section~5.4]{fujishige2022discrete}.}
Fix a compatible family of reflecting sets
\(\mathcal T=(T_z)_{z\in\mathbb Z^n}\).
For \(x,y\in\mathbb Z^n\), put
\(m:=(x+y)/2\), \(m^-:=\lfloor m\rfloor\), and
\(m^+:=\lceil m\rceil\), and define
\(\mu_1(x,y),\mu_2(x,y)\in\mathbb Z^n\) by
\begin{equation}
\mu_1(x,y)_i
:=
\begin{cases}
m_i^+, & i\in[n]\setminus T_{m^-},\\
m_i^-, & i\in T_{m^-},
\end{cases}
\qquad
\mu_2(x,y)_i
:=
\begin{cases}
m_i^-, & i\in[n]\setminus T_{m^-},\\
m_i^+, & i\in T_{m^-},
\end{cases}
\quad
(i\in[n]).
\label{eq:d-midpoint-operators}
\end{equation}

A nonempty set \(S\subseteq\mathbb Z^n\) is called
\emph{\(\mathrm{D}\)-convex with respect to \(\mathcal T\)} if
\[
x,y\in S
\quad\Longrightarrow\quad
\mu_1(x,y),\ \mu_2(x,y)\in S.
\]
Similarly, a function
\(f\colon\mathbb Z^n\to\mathbb R\cup\{+\infty\}\)
is called \emph{\(\mathrm{D}\)-convex with respect to \(\mathcal T\)} if
\[
f(x)+f(y)
\geq
f\bigl(\mu_1(x,y)\bigr)
+
f\bigl(\mu_2(x,y)\bigr)
\qquad
(x,y\in\mathbb Z^n).
\]
A set or function is called
\emph{\(\mathrm{D}\)-convex on the integer lattice}
if it is \(\mathrm{D}\)-convex with respect to some compatible family
\(\mathcal T\).

We next specialize the above definition to a unit hypercube.
Since compatibility with neighboring unit hypercubes is no longer
involved, the reflected structure is determined by a single subset
\(T\subseteq[n]\).
In this setting, the midpoint conditions defining
\(\mathrm{D}\)-convexity can equivalently be expressed in terms of
reflected meet and join operations induced by a coordinate reflection.
For \(x,y\in\{0,1\}^n\), let
\begin{equation}
(x\wedge y)_i:=\min\{x_i,y_i\},
\qquad
(x\vee y)_i:=\max\{x_i,y_i\}
\qquad (i\in[n]).
\label{eq:componentwise-meet-join}
\end{equation}
Under the identification \(x=\chi_A\), \(y=\chi_B\), where
\(A,B\subseteq[n]\), these operations correspond to intersection and union:
\(
\chi_A\wedge\chi_B=\chi_{A\cap B}
\)
and
\(
\chi_A\vee\chi_B=\chi_{A\cup B}.
\)

For \(T\subseteq[n]\), let
\(\rho_T:\{0,1\}^n\to\{0,1\}^n\) be the coordinate reflection
defined by
\begin{equation}
\rho_T(x)_i
:=
\begin{cases}
1-x_i, & i\in T,\\
x_i, & i\notin T.
\end{cases}
\label{eq:coordinate-reflection-on-unit-hypercube}
\end{equation}
It follows directly from
\eqref{eq:coordinate-reflection-on-unit-hypercube} that
\(\rho_T^2=\operatorname{id}_{\{0,1\}^n}\).
Thus, \(\rho_T\) is an involution and
\(\rho_T^{-1}=\rho_T\).

The componentwise meet and join in
\eqref{eq:componentwise-meet-join} induce the reflected meet and
join associated with \(T\):
\begin{equation}
x\wedge_Ty
:=
\rho_T\bigl(\rho_T(x)\wedge\rho_T(y)\bigr),
\qquad
x\vee_Ty
:=
\rho_T\bigl(\rho_T(x)\vee\rho_T(y)\bigr).
\label{eq:reflected-meet-join}
\end{equation}
By \eqref{eq:reflected-meet-join} and
\(\rho_T^2=\operatorname{id}_{\{0,1\}^n}\), we have
\(
\rho_T(x\wedge_Ty)=\rho_T(x)\wedge\rho_T(y)
\)
and
\(
\rho_T(x\vee_Ty)=\rho_T(x)\vee\rho_T(y).
\)
Hence \((\{0,1\}^n,\wedge_T,\vee_T)\) is a distributive lattice,
and \(\rho_T\) is a lattice isomorphism from
\((\{0,1\}^n,\wedge_T,\vee_T)\) to
\((\{0,1\}^n,\wedge,\vee)\).

Using this distributive lattice representation, the above definition of
\(\mathrm{D}\)-convexity specializes as follows on a unit hypercube.
For \(T\subseteq[n]\), a nonempty set
\(X\subseteq\{0,1\}^n\) is \(\mathrm{D}\)-convex with respect to \(T\)
if and only if
\begin{equation}
x,y\in X
\quad\Longrightarrow\quad
x\wedge_Ty,\ x\vee_Ty\in X.
\label{eq:d-convex-set-unit-hypercube}
\end{equation}
Similarly, a function
\(f\colon\{0,1\}^n\to\mathbb R\cup\{+\infty\}\)
is \(\mathrm{D}\)-convex with respect to \(T\) if and only if
\begin{equation}
f(x)+f(y)
\geq
f(x\wedge_Ty)+f(x\vee_Ty)
\qquad
(x,y\in\{0,1\}^n).
\label{eq:d-convex-function-unit-hypercube}
\end{equation}

We next recall its relation to the
\(\frac12\mathrm{UJ}\) simplicial division.
The following refinement result is due to
Fujishige~\cite[Theorem~9]{fujishige2014bisubmodular}.

\begin{proposition}[
Refinement by the \(\frac12\mathrm{UJ}\) simplicial division{~\cite[Theorem~9]{fujishige2014bisubmodular}}]
\label{prop:discordant-freudenthal-refinement}
The \(\frac12\mathrm{UJ}\) simplicial division of \(\mathbb R^n\) refines every discordant Freudenthal simplicial division of \(\mathbb R^n\).
\end{proposition}

Locally on each unit hypercube, this refinement relation can be seen directly from the signed-chain description in \eqref{eq:half-uj-simplex-signed-chain}.
For any choice of \(s_i\in\{e_i,-e_i\}\) \((i\in[n])\)
and a permutation \(\pi\) of \([n]\), the corresponding
\(\frac12\mathrm{UJ}\)-simplex is contained in an
\(n\)-dimensional simplex of the reflected Freudenthal division on that unit hypercube.

\section{A Necessary Condition for \(\mathrm{L}\)-Extendability}
\label{sec:l-extendability-integral-convexity}

We prove that every $\mathrm{L}$-extendable function
on the integer lattice is integrally convex.
This necessary condition holds for arbitrary
$\mathrm{L}$-convex relaxations, whose values at nonintegral
half-integral points are subject only to
$\frac12\mathrm{UJ}$-convexity.

We use the following characterization of integral convexity due to Moriguchi, Murota, Tamura, and Tardella.

\begin{lemma}[{\cite[Theorem~A.1]{moriguchi2020discrete}}]
\label{lem:integral-convexity-global-weak-dmc}
Let \(f:\mathbb Z^n\to\mathbb R\cup\{+\infty\}\), and let
\(\tilde f\) denote its local convex extension.  Then the following
conditions are equivalent:
\begin{enumerate}
\item[\textup{(i)}]
\(f\) is integrally convex.
\item[\textup{(ii)}]
For every \(x,y\in\operatorname{dom} f\) with
\(\lVert x-y\rVert_\infty\ge 2\),
\begin{equation}
f(x)+f(y)
\ge
2\tilde f\left(\frac{x+y}{2}\right).
\label{eq:global-weak-discrete-midpoint-inequality}
\end{equation}
\end{enumerate}
\end{lemma}

Using this lemma, we prove the following.

\begin{theorem}[\(\mathrm{L}\)-extendability implies integral convexity]
\label{thm:l-extendability-implies-integral-convexity}
Let \(f:\mathbb Z^n\to\mathbb R\cup\{+\infty\}\).
If \(f\) is \(\mathrm{L}\)-extendable, then \(f\) is integrally convex.
\end{theorem}

\begin{pf}
Let \(\tilde f\) denote the local convex extension of \(f\).
Since \(f\) is \(\mathrm{L}\)-extendable, there exists a
\(\frac12\mathrm{UJ}\)-convex extension
\(g:(\frac12\mathbb Z)^n\to\mathbb R\cup\{+\infty\}\)
of \(f\).
Thus, \(g(x)=f(x)\) for every \(x\in\mathbb Z^n\).

By Lemma~\ref{lem:integral-convexity-global-weak-dmc}, it suffices to
prove \eqref{eq:global-weak-discrete-midpoint-inequality} for all
\(x,y\in\dom f\) with \(\lVert x-y\rVert_\infty\geq2\).
In fact, we prove it for all \(x,y\in\mathbb Z^n\).
Take \(x,y\in\mathbb Z^n\), and put \(z:=(x+y)/2\).
If \(f(x)=+\infty\) or \(f(y)=+\infty\), the inequality is immediate.
Thus we may assume that \(f(x)\) and \(f(y)\) are finite.  Since
\(g(x)=f(x)\) and \(g(y)=f(y)\), the values \(g(x)\) and \(g(y)\) are
finite.  Applying the \(\frac12\mathrm{UJ}\)-convexity inequality to
\(x\) and \(y\), and noting that
\[
\left[ z\right]_{\frac12}
=
\left[ z\right]_{1}
=
z,
\]
we obtain \(g(x)+g(y)\ge 2g(z)\).  In particular,
\(g(z)<+\infty\).

Put
\(
a:=(x+z)/{2},
~b:=(z+y)/{2}
\),
and define
\[
a^-:=\left[ a\right]_{\frac12},\quad
a^+:=\left[a\right]_{1},\quad
b^-:=\left[ b\right]_{\frac12},\quad
b^+:=\left[b\right]_{1},\qquad
r:=\frac{a^++b^-}{2},\quad
s:=\frac{a^-+b^+}{2}.
\]

We use the following rounding relations:
\[
\textup{(i)}\quad
\left[ r\right]_{\frac12}
=
\left[ s\right]_{\frac12}
=
z;
\qquad
\textup{(ii)}\quad
\left[ r\right]_{1},
\left[ s\right]_{1}
\in\mathbb Z^n\cap N(z);
\qquad
\textup{(iii)}\quad
\left[ r\right]_{1}
+
\left[ s\right]_{1}
=
2z.
\]
The verification of these relations is given in
Appendix~\ref{app:rounding-relations-l-extendability};
see Lemma~\ref{lem:rounding-relations-for-l-extendability}.

We now apply the \(\frac12\mathrm{UJ}\)-convexity inequality four times.
Applying it to \(x,z\) and to \(z,y\), respectively, gives
\begin{align}
g(x)+g(z)
&\ge
g(a^-)+g(a^+),
\label{eq:midpoint-xz-general-relaxation}\\
g(z)+g(y)
&\ge
g(b^-)+g(b^+).
\label{eq:midpoint-zy-general-relaxation}
\end{align}
Since the left-hand sides are finite, all four values on the right-hand
sides are finite.

Applying the same inequality to \(a^+,b^-\), whose midpoint is \(r\),
and using~\textup{(i)}, we
obtain
\begin{equation}
g(a^+)+g(b^-)
\ge
g(z)+
g\left(
\left[ r\right]_{1}
\right).
\label{eq:midpoint-cross-r-general-relaxation}
\end{equation}
Similarly, applying it to \(a^-,b^+\), whose midpoint is \(s\), gives
\begin{equation}
g(a^-)+g(b^+)
\ge
g(z)+
g\left(
\left[ s\right]_{1}
\right).
\label{eq:midpoint-cross-s-general-relaxation}
\end{equation}

All terms in
\eqref{eq:midpoint-xz-general-relaxation}--%
\eqref{eq:midpoint-cross-s-general-relaxation}
are finite.  Adding these four inequalities and cancelling the common
finite terms, we obtain
\begin{equation}
g(x)+g(y)
\ge
g\left(
\left[ r\right]_{1}
\right)
+
g\left(
\left[ s\right]_{1}
\right).
\label{eq:gx-gy-geq-ceil-rs}
\end{equation}

By \textup{(ii)}, both
\([ r]_{1}\) and
\([s]_{1}\) belong to
\(\mathbb Z^n\cap N(z)\); hence \(g\) agrees with \(f\) at both points.
Moreover, by \textup{(iii)},
\[
z
=
\frac12
\left[ r\right]_{1}
+
\frac12
\left[ s\right]_{1}.
\]
Therefore, by the definition of the local convex extension,
\[
\begin{aligned}
\tilde f(z)
&\le
\frac12
f\left(
\left[ r\right]_{1}
\right)
+
\frac12
f\left(
\left[ s\right]_{1}
\right)\\
&=
\frac12
g\left(
\left[ r\right]_{1}
\right)
+
\frac12
g\left(
\left[ s\right]_{1}
\right)\\
&\le
\frac{g(x)+g(y)}{2}
=
\frac{f(x)+f(y)}{2},
\end{aligned}
\]
where the last inequality follows from \eqref{eq:gx-gy-geq-ceil-rs}.
Since \(z=(x+y)/2\), this proves
\[
f(x)+f(y)
\ge
2\tilde f\left(\frac{x+y}{2}\right)
\qquad
(x,y\in\mathbb Z^n).
\]
Thus \(f\) satisfies global weak discrete midpoint convexity and is
therefore integrally convex.\qed
\end{pf}

Theorem~\ref{thm:l-extendability-implies-integral-convexity} establishes
the inclusion
\[
\{\text{\(\mathrm{L}\)-extendable functions}\}
\subseteq
\{\text{integrally convex functions}\}.
\]
This inclusion is strict, as will be shown in
Example~\ref{ex:relative-opposite-not-majority} in Section~\ref{subsec:set-theoretic-unit-hypercube}.

A complete characterization of the integrally convex functions that are \(\mathrm{L}\)-extendable remains open.

\section{Characterizations of \(\mathrm{L}\)-Extendability through Local Convex Extensions}
\label{sec:local-structures}

We characterize functions on the integer lattice
that are $\mathrm{L}$-extendable through their local convex
extensions in terms of integral convexity and
$\mathrm{D}$-convexity of their unit hypercube restrictions.
Let \(f:\mathbb Z^n\to\mathbb R\cup\{+\infty\}\), and let \(\tilde f:\mathbb R^n\to\mathbb R\cup\{+\infty\}\) be its local convex extension.
We say that \(f\) is \emph{\(\mathrm{L}\)-extendable (or \(\frac12\mathrm{UJ}\)-extendable) through its local convex extension} if the restriction of \(\tilde f\) to \((\frac12\mathbb Z)^n\) is \(\frac12\mathrm{UJ}\)-convex.

\subsection{Subsets of the vertices of a unit hypercube}
\label{subsec:set-theoretic-unit-hypercube}

As a first step toward characterizing $\mathrm{L}$-extendability
through local convex extensions, we consider indicator functions
of subsets of the vertices of a unit hypercube.
For a nonempty subset \(X\subseteq\{0,1\}^n\), we show that its
indicator function is \(\mathrm{L}\)-extendable through its local convex extension if and only if \(X\) is \(\mathrm{D}\)-convex.
We further show that this condition is equivalent to \(X\) being majority-closed and containing a relative opposite pair.
Thus, as a byproduct, we obtain a combinatorial characterization of the known discrete convexity class of \(\mathrm{D}\)-convex subsets of the vertices of a unit hypercube.

We say that $\overline X$ is
\emph{$\frac12\mathrm{UJ}$-triangulable}
if it admits a simplicial division consisting entirely of
$\frac12\mathrm{UJ}$-simplices.
The following lemma relates this property to
$\mathrm{L}$-extendability through the local convex extension
and is a consequence of results of
Iwamasa~\cite{iwamasa2026towards} on UJ-convex sets
at the integer scale.
We explain the correspondence with the half-integral
setting used here in the proof.

\begin{lemma}
\label{lem:halfUJ-triangulable-iff-convex}
Let \(X\subseteq\{0,1\}^n\) be nonempty, and let
\(\delta_X:\mathbb Z^n\to\mathbb R\cup\{+\infty\}\)
denote its indicator function.
Then the following conditions are equivalent:
\begin{enumerate}
\item[\textup{(i)}]
The function \(\delta_X\) is \(\mathrm{L}\)-extendable through its
local convex extension.
\item[\textup{(ii)}]
\(\overline X\) is \(\frac12\mathrm{UJ}\)-triangulable.
\end{enumerate}
\end{lemma}

\begin{pf}
Let \(\widetilde{\delta_X}\) denote the local convex extension of
\(\delta_X\).
The effective domain of
\(\widetilde{\delta_X}|_{(\frac12\mathbb Z)^n}\) is
\(
\overline X\cap(\frac12\mathbb Z)^n.
\)
Since \(\widetilde{\delta_X}|_{(\frac12\mathbb Z)^n}\) is the
indicator function of this set, it is
\(\frac12\mathrm{UJ}\)-convex if and only if
\(
\overline X\cap(\frac12\mathbb Z)^n
\)
is \(\frac12\mathrm{UJ}\)-convex.
Hence, by the definition of \(\mathrm{L}\)-extendability through the
local convex extension, condition \textup{(i)} is equivalent to this
set being \(\frac12\mathrm{UJ}\)-convex.

\(\mathrm{UJ}\)-convex sets on the integer lattice are $\mathrm{L}$-convex sets of type $C_n$ in Iwamasa's terminology~\cite[Section~8.5 and Corollary~B.5]{iwamasa2026towards}; see also~\cite{hirai2018convexity}.
Moreover, Iwamasa provides the following characterization of
\(\mathrm{UJ}\)-convex sets: a nonempty set
\(X\subseteq\mathbb Z^n\) is \(\mathrm{UJ}\)-convex if and only if
\(X=\overline X\cap\mathbb Z^n\) and \(\overline X\) is the union of the simplices of a nonempty convex subcomplex of the Euclidean Coxeter
complex of type \(C_n\)
\cite[Lemma~5.5 and Sections~6.1 and~8.5]{iwamasa2026towards}.

We now pass from the integer scale used by Iwamasa~\cite{iwamasa2026towards} to the
half-integral scale used in this paper.
Since
\(
X
\subseteq
\overline X\cap(\frac12\mathbb Z)^n
\subseteq
\overline X,
\)
the convex hull of
\(
\overline X\cap(\frac12\mathbb Z)^n
\)
is \(\overline X\).
Under the half-integral rescaling used here, UJ-convexity becomes
\(\frac12\mathrm{UJ}\)-convexity, and the simplicial complex
underlying the UJ structure becomes the
\(\frac12\mathrm{UJ}\) simplicial division.
Consequently, the results cited above say precisely that
\(
\overline X\cap(\frac12\mathbb Z)^n
\)
is \(\frac12\mathrm{UJ}\)-convex if and only if
\(\overline X\) is the union of the simplices of a subcomplex of
the \(\frac12\mathrm{UJ}\) simplicial division; this is exactly the condition
that \(\overline X\) be \(\frac12\mathrm{UJ}\)-triangulable.
\qed
\end{pf}

We now introduce the notion of a relative opposite pair.

\begin{definition}[Relative opposite pairs]
\label{def:relative-opposite-pair}
Let \(X\subseteq\{0,1\}^n\) be nonempty.
Define
\begin{equation}
I_X
:=
\{i\in[n]\mid \text{there exist } x,y\in X \text{ such that } x_i\neq y_i\}.
\label{eq:relative-opposite-free-coordinates}
\end{equation}
For each \(i\in[n]\setminus I_X\), all points of \(X\) have the same \(i\)-th coordinate; denote this common value by \(\alpha_i\in\{0,1\}\).
Then the smallest face of \([0,1]^n\) containing \(X\) is
\begin{equation}
F_X
=
\{u\in[0,1]^n\mid u_i=\alpha_i \ (i\in[n]\setminus I_X)\}.
\label{eq:smallest-face-containing-x}
\end{equation}

For \(d\in X\), define its \emph{relative opposite point} \(d^{\operatorname{opp}}\) in \(F_X\) by
\begin{equation}
d^{\operatorname{opp}}_i
:=
\begin{cases}
1-d_i, & i\in I_X,\\
d_i, & i\notin I_X.
\end{cases}
\label{eq:relative-opposite-point}
\end{equation}
If \(d^{\operatorname{opp}}\in X\), then \(d\) and
\(d^{\operatorname{opp}}\) are called a
\emph{relative opposite pair} of \(X\).
If \(I_X=\emptyset\), or equivalently if \(X\) is a singleton, then
\(d^{\operatorname{opp}}=d\) for its unique point \(d\).
Thus every singleton subset of \(\{0,1\}^n\) contains a relative opposite pair.
\end{definition}

The following lemma shows that
\(\frac12\mathrm{UJ}\)-triangulability implies the existence of a relative opposite pair.

\begin{lemma}
\label{lem:halfUJ_triangulable_implies_relative_opposite_pair}
Let \(X\subseteq\{0,1\}^n\) be nonempty.
If \(\overline X\) is \(\frac12\mathrm{UJ}\)-triangulable, then \(X\) contains a relative opposite pair.
\end{lemma}

\begin{pf}
If \(I_X=\emptyset\), then \(X\) is a singleton, and its unique
element is its own relative opposite point.
Hence the assertion is immediate.
We may therefore assume that \(I_X\neq\emptyset\).

For \(i\notin I_X\), let \(\alpha_i\in\{0,1\}\) be the fixed
\(i\)-th coordinate of \(X\), and define
\(c_X\in(\frac12\mathbb Z)^n\) by
\[
(c_X)_i=
\begin{cases}
\frac12, & i\in I_X,\\
\alpha_i, & i\notin I_X.
\end{cases}
\]
We first show that \(c_X\in\overline X\).
Since \(\overline X\) is \(\frac12\mathrm{UJ}\)-triangulable,
Lemma~\ref{lem:halfUJ-triangulable-iff-convex} implies that the
indicator function of \(X\) is \(\mathrm{L}\)-extendable through its
local convex extension.
Hence, the restriction of its local convex extension to
\((\frac12\mathbb Z)^n\) is \(\frac12\mathrm{UJ}\)-convex.
Its effective domain is
\(
\overline X\cap(\frac12\mathbb Z)^n,
\)
and therefore this set is closed under the
\(\frac12\mathrm{UJ}\)-rounding operations.
For \(a,b\in\{0,\frac12,1\}\), the definition of the
\(\frac12\mathrm{UJ}\)-rounding gives
\[
\left[\frac{a+b}{2}\right]_{\frac12}
=
\begin{cases}
a, & a=b,\\
\frac12, & a\neq b.
\end{cases}
\]
Hence, by repeatedly applying this operation to the points of
\(X\), every fixed coordinate remains unchanged, whereas every
coordinate in \(I_X\) becomes \(\frac12\).
We thus obtain
\(c_X\in\overline X\).

We next find a point of \(X\) whose relative opposite point also
belongs to \(X\).
Let \(Q\) be the minimal face of \(\overline X\) containing
\(c_X\), and choose a vertex \(y\) of \(Q\).
Then \(c_X\in\operatorname{relint}Q\).
Moreover, since \(Q\) is a face of \(\overline X\), its vertices
are vertices of \(\overline X\), and hence \(y\in X \subseteq \{0,1\}^n \).
In particular, we have $c_X \neq y$ and $\dim Q \ge 1$.

Let \(y^{\operatorname{opp}}\) be the relative opposite point of
\(y\).
By the definition of \(c_X\),
\(
y^{\operatorname{opp}}=2c_X-y.
\)
Since \(c_X\in\operatorname{relint}Q\) and $y \in Q$, for sufficiently small
\(0<\varepsilon<1\),
\[
(1-\varepsilon)c_X+\varepsilon y^{\operatorname{opp}}
=
c_X-\varepsilon(y-c_X)
\in Q\subseteq\overline X.
\]
Therefore,
\(
\operatorname{relint}
[c_X,y^{\operatorname{opp}}]\cap\overline X
\neq\emptyset.
\)

We now use the \(\frac12\mathrm{UJ}\)-triangulation to show that
the endpoint \(y^{\operatorname{opp}}\) also belongs to
\(\overline X\).
To this end, we first show that
\(
E:=[c_X,y^{\operatorname{opp}}]
\)
is an edge of the \(\frac12\mathrm{UJ}\) simplicial division.
In \eqref{eq:half-uj-simplex-signed-chain}, choose
\[
s_i:=
\begin{cases}
(2\alpha_i-1)e_i, & i\notin I_X,\\
(2y_i^{\operatorname{opp}}-1)e_i, & i\in I_X,
\end{cases}
\]
and choose the permutation \(\pi\) so that the coordinates in
\([n]\setminus I_X\) precede those in \(I_X\).
Then
\[
c+\frac12\sum_{i\notin I_X}s_i=c_X,
\qquad
c+\frac12\sum_{i=1}^n s_i=y^{\operatorname{opp}},
\]
where \(c=\frac12\mathbf 1\).
Hence \(c_X\) and \(y^{\operatorname{opp}}\) are vertices of the
same \(n\)-dimensional \(\frac12\mathrm{UJ}\)-simplex.
Therefore \(E\) is an edge of the
\(\frac12\mathrm{UJ}\) simplicial division.

Choose
\(u\in\operatorname{relint}E\cap\overline X\), and let
\(\Delta\) be a simplex in a
\(\frac12\mathrm{UJ}\)-triangulation of \(\overline X\) that
contains \(u\).
Since \(E\) and \(\Delta\) are simplices of the same
\(\frac12\mathrm{UJ}\) simplicial division,
\(E\cap\Delta\) is a common face.
As this common face contains a point of
\(\operatorname{relint}E\), it cannot be a face of \(E\).
Thus \(E\subseteq\Delta\), and hence
\(y^{\operatorname{opp}}\in\overline X\).

Finally,
\(y^{\operatorname{opp}}\in\{0,1\}^n\).
Since
\(
\overline X\cap\{0,1\}^n=X
\),
we obtain \(y^{\operatorname{opp}}\in X\).
Thus \(y,y^{\operatorname{opp}}\) form a relative opposite pair.
\qed
\end{pf}

We are now ready to state a characterization of \(\frac12\mathrm{UJ}\)-triangulability for subsets of the vertices of a unit hypercube.

\begin{theorem}[Subsets of the vertices of a unit hypercube]
\label{thm:subsets-of-unit-hypercube}
Let \(X\subseteq\{0,1\}^n\) be nonempty, and let
\(\delta_X:\mathbb Z^n\to\mathbb R\cup\{+\infty\}\)
denote its indicator function.
Then the following conditions are equivalent.
\begin{enumerate}
\item[\textup{(i)}]
The function \(\delta_X\) is \(\mathrm{L}\)-extendable through its local convex extension.

\item[\textup{(ii)}]
\(\overline X\) is \(\frac12\mathrm{UJ}\)-triangulable.

\item[\textup{(iii)}]
\(X\) is a \(\mathrm{D}\)-convex set.

\item[\textup{(iv)}]
\(X\) is majority-closed and contains a relative opposite pair.
\end{enumerate}
\end{theorem}

\begin{pf}
The equivalence \(\textup{(i)}\Leftrightarrow\textup{(ii)}\) follows from
Lemma~\ref{lem:halfUJ-triangulable-iff-convex}.
It remains to prove
\(\textup{(iii)}\Rightarrow\textup{(ii)}
\Rightarrow\textup{(iv)}
\Rightarrow\textup{(iii)}\).

We first prove \(\textup{(iii)}\Rightarrow\textup{(ii)}\). 
Suppose that \(X\) is \(\mathrm{D}\)-convex.
Then \(X\) is compatible with a reflected Freudenthal simplicial division.
By Proposition~\ref{prop:discordant-freudenthal-refinement}, the \(\frac12\mathrm{UJ}\) simplicial division refines this division.
Hence \(\overline X\) is \(\frac12\mathrm{UJ}\)-triangulable.

We next prove \(\textup{(ii)}\Rightarrow\textup{(iv)}\). 
Suppose that \(\overline X\) is \(\frac12\mathrm{UJ}\)-triangulable.
By Lemma~\ref{lem:halfUJ-triangulable-iff-convex}, the indicator function \(\delta_X\) is \(\mathrm{L}\)-extendable through its local convex extension,
and hence is \(\mathrm{L}\)-extendable.
Lemma~\ref{cor:hirai-iwamasa-majority-closure} therefore implies that \(X=\dom\delta_X\) is majority-closed.
Moreover, by Lemma~\ref{lem:halfUJ_triangulable_implies_relative_opposite_pair}, \(X\) contains a relative opposite pair.
Thus \(\textup{(iv)}\) holds.

Finally, we prove
\(\textup{(iv)}\Rightarrow\textup{(iii)}\).
Choose a relative opposite pair
\(d,d^{\operatorname{opp}}\in X\), and put
\(
T:=\{i\in I_X\mid d_i=1\}.
\)
Since coordinate reflections commute with the majority operation, \(Y:=\rho_T(X)\) is majority-closed.
Moreover, \(\rho_T(d)\) has value \(0\) and \(\rho_T(d^{\operatorname{opp}})\) has value \(1\) in every coordinate in \(I_X\), while every coordinate outside \(I_X\) is fixed on \(Y\).
Thus, for \(x,y\in Y\),
\[
x\wedge y
=
\operatorname{maj}(x,y,\rho_T(d)),
\qquad
x\vee y
=
\operatorname{maj}(x,y,\rho_T(d^{\operatorname{opp}})).
\]
Since \(Y\) is majority-closed,
\(x\wedge y,x\vee y\in Y\).
Hence \(Y=\rho_T(X)\) is closed under componentwise meet and join, and therefore \(X\) is \(\mathrm{D}\)-convex.
\qed
\end{pf}

Figure~\ref{fig:local-set-classes} illustrates Theorem~\ref{thm:subsets-of-unit-hypercube}.
Examples~\ref{ex:standard-simplex} and~\ref{ex:relative-opposite-not-majority}
show that the conditions of being majority-closed and
containing a relative opposite pair are independent.

\begin{figure}
\centering
\includegraphics[width=.80\linewidth]{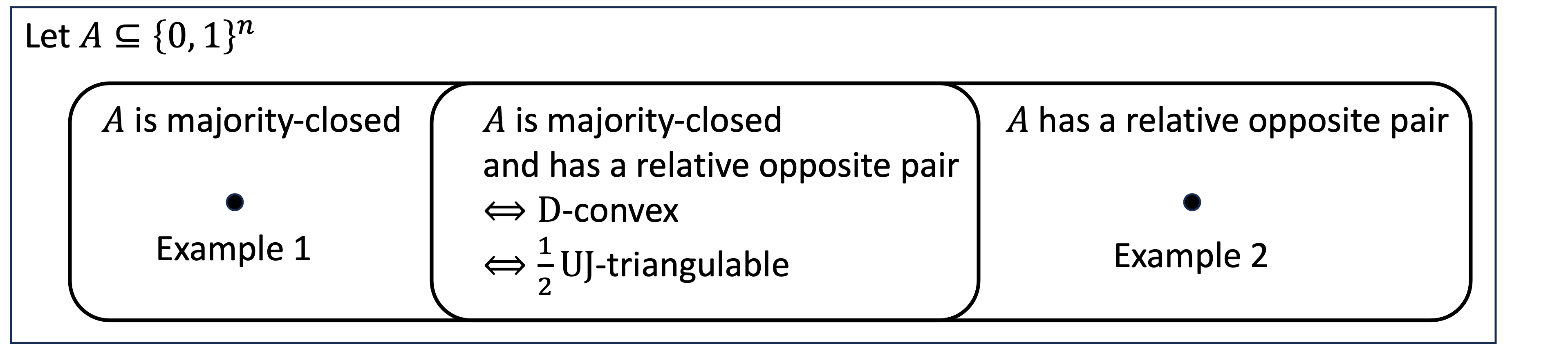}
\caption{
Relations among local set classes for subsets of \(\{0,1\}^n\).}
\label{fig:local-set-classes}
\end{figure}

\begin{example}
\label{ex:standard-simplex}
Let
\(
X_1
:=
\{(0,0,0),(1,0,0),(0,1,0),(0,0,1)\}
\subseteq\{0,1\}^3.
\)
This is the vertex set of the standard simplex.
The set \(X_1\) is majority-closed.
On the other hand, we have \(I_{X_1}=\{1, 2, 3\}\).
Hence the relative opposite point of \(d\in X_1\) is obtained by reversing all coordinates of \(d\).
Thus the relative opposite points of
\((0,0,0),(1,0,0),(0,1,0)\), and \((0,0,1)\) are
\((1,1,1),(0,1,1),(1,0,1)\), and \((1,1,0)\), respectively.
None of them belongs to \(X_1\).
Therefore \(X_1\) has no relative opposite pair.
Consequently, by Theorem~\ref{thm:subsets-of-unit-hypercube}, \(X_1\) is not \(\frac12\mathrm{UJ}\)-triangulable.
\end{example}

\begin{example}
\label{ex:relative-opposite-not-majority}
Let
\(
X_2
:=
\{(0,0,0),(1,0,0),(0,1,0),(1,1,1)\}
\subseteq\{0,1\}^3\).
We have \(I_{X_2}=\{1, 2, 3\}\).
Hence the relative opposite point of \(d\in X_2\) is obtained by reversing all coordinates of \(d\).
In particular, the relative opposite point of \((0,0,0)\) is \((1,1,1)\), which belongs to \(X_2\).
Thus \(X_2\) contains a relative opposite pair.
However, \(X_2\) is not majority-closed.
Indeed,
\(
\operatorname{maj}\bigl((1,0,0),(0,1,0),(1,1,1)\bigr)
=
(1,1,0)
\notin X_2\).
Thus, by Theorem~\ref{thm:subsets-of-unit-hypercube}, \(X_2\) is not \(\frac12\mathrm{UJ}\)-triangulable.
\end{example}

Every subset of \(\{0,1\}^n\) is integrally convex;
see \cite[Example~2.2]{murota2023recent}.
Hence the indicator function
\(\delta_{X_2}:\mathbb Z^3\to\mathbb R\cup\{+\infty\}\)
is integrally convex.
Since \(X_2\) is not majority-closed,
Lemma~\ref{cor:hirai-iwamasa-majority-closure} implies that
\(\delta_{X_2}\) is not \(\mathrm{L}\)-extendable.
Thus the inclusion in
Theorem~\ref{thm:l-extendability-implies-integral-convexity}
is strict.

\subsection{Functional structure on a unit hypercube}
\label{subsec:functional-unit-hypercube}

We now extend the characterization for indicator functions
to arbitrary functions on the vertices of a unit hypercube.
In preparation for this extension, the following lemma shows
that a single coordinate reflection can be chosen for
the minimizer sets of all linear perturbations.

\begin{lemma}
\label{lem:common_reflection_for_all_argmins}
Let \(f:\{0,1\}^n\to\mathbb R\cup\{+\infty\}\), and let
\(\tilde f\) be its local convex extension.
Assume that
\(\tilde f|_{\{0,\frac12,1\}^n}\) is
\(\frac12\mathrm{UJ}\)-convex.
For \(p\in\mathbb R^n\), put
\(A_p:=\operatorname*{argmin}_{x\in\{0,1\}^n}f[-p](x)\).
Then there exists \(T\subseteq[n]\) such that, for every
\(p\in\mathbb R^n\), the set \(\rho_T(A_p)\) is closed under
componentwise meet and join, where \(\rho_T\) is the coordinate
reflection defined by
\eqref{eq:coordinate-reflection-on-unit-hypercube}.
\end{lemma}

\begin{pf}
Regard \(\tilde f\) as taking the value \(+\infty\) outside
\([0,1]^n\).
Since \(\tilde f\) is the convex envelope of \(f\) on the finite
set \(\{0,1\}^n\), it is a polyhedral convex function.
Moreover, \(\dom\tilde f=\overline{\dom f}\) and
\(\argmin_{x\in\mathbb R^n}\tilde f[-p](x)=\overline{A_p}\)
for every \(p\in\mathbb R^n\).

Linear perturbations preserve \(\frac12\mathrm{UJ}\)-convexity.
Thus, for every \(p\in\mathbb R^n\),
\(\tilde f[-p]|_{\{0,\frac12,1\}^n}\) is
\(\frac12\mathrm{UJ}\)-convex.
Since \(A_p\) is nonempty and consists of integer points, the above
identity implies that the minimizer set of this restriction is
\(\overline{A_p}\cap\{0,\frac12,1\}^n\), which is itself
\(\frac12\mathrm{UJ}\)-convex.
Lemma~\ref{lem:halfUJ-triangulable-iff-convex} then implies that
\(A_p\) is \(\frac12\mathrm{UJ}\)-triangulable.
Hence, Theorem~\ref{thm:subsets-of-unit-hypercube} shows that
every \(A_p\) is majority-closed and contains a relative opposite pair.

The affinity domains of a polyhedral convex function form a
polyhedral complex whose support is its effective domain
\cite[Sections~2.3.4--2.3.5]{iwamasa2026towards}.
Consequently,
\(\{\overline{A_p}\mid p\in\mathbb R^n\}\)
is a finite polyhedral complex with support \(\overline{\dom f}\).

Let
\(\overline{A_{p_1}},\ldots,\overline{A_{p_m}}\)
be the maximal members of this complex.
Since the support \(\overline{\dom f}\) is convex, every maximal
member has dimension \(\dim\overline{\dom f}\).
Since \(A_{p_k}\subseteq\dom f\), it follows that
\(\operatorname{aff}\overline{A_{p_k}}
=\operatorname{aff}\overline{\dom f}\) for every \(k\in[m]\).
A coordinate is fixed on a set if and only if it is fixed on its
affine hull; hence
\(I_{A_{p_k}}=I_{\dom f}\) for every \(k\in[m]\).

We now show that these maximal affinity domains contain a common
point determined only by \(\dom f\).
For every \(i\in[n]\setminus I_{\dom f}\), let
\(\alpha_i\in\{0,1\}\) be the fixed \(i\)-th coordinate of
\(\dom f\), and define \(c\in[0,1]^n\) by
\[
c_i=
\begin{cases}
\frac12, & i\in I_{\dom f},\\
\alpha_i, & i\in[n]\setminus I_{\dom f}.
\end{cases}
\]
As shown above, each \(A_{p_k}\) contains a relative opposite pair.
Since \(I_{A_{p_k}}=I_{\dom f}\), the midpoint of such a pair is \(c\).
Thus \(c\in\overline{A_{p_k}}\) for every \(k\in[m]\), and hence
\(c\in\bigcap_{k=1}^m\overline{A_{p_k}}\).

We next realize this common intersection as a single affinity domain.
Since the affinity domains form a polyhedral complex and their
intersection is nonempty, there exists \(q\in\mathbb R^n\) such that
\(\overline{A_q}
=\bigcap_{k=1}^m\overline{A_{p_k}}\).
In particular, \(c\in\overline{A_q}\).

We now determine the nonfixed coordinates of \(A_q\).
Since \(A_q\subseteq\dom f\), we have
\(I_{A_q}\subseteq I_{\dom f}\).
Conversely, for every \(i\in I_{\dom f}\), the point
\(c\in\overline{A_q}\) satisfies \(c_i=\frac12\).
If the \(i\)-th coordinate were fixed on \(A_q\), then it would be
fixed to either \(0\) or \(1\) throughout \(\overline{A_q}\), a
contradiction.
Thus \(I_{A_q}=I_{\dom f}\).

Choose a relative opposite pair
\(d,d^{\operatorname{opp}}\in A_q\).
For every \(k\in[m]\), we have
\(\overline{A_q}\subseteq\overline{A_{p_k}}\).
Since \(d,d^{\operatorname{opp}}\in\{0,1\}^n\), this implies
\(d,d^{\operatorname{opp}}\in A_{p_k}\).
Hence, as \(I_{A_q}=I_{A_{p_k}}=I_{\dom f}\),\(d,d^{\operatorname{opp}}\) is a relative opposite pair of every
\(A_{p_k}\).

Set \(T:=\{i\in I_{\dom f}\mid d_i=1\}\).
We show that the same reflection \(T\) works for all maximal
affinity domains.
Coordinate reflections commute with the majority operation,
and hence \(\rho_T(A_{p_k})\) is majority-closed for every
\(k\in[m]\).
Moreover, \(\rho_T(d)\) has value \(0\) and
\(\rho_T(d^{\operatorname{opp}})\) has value \(1\) in every
coordinate in \(I_{\dom f}\), while every coordinate outside
\(I_{\dom f}\) is fixed on \(\rho_T(A_{p_k})\).
Thus, for \(x,y\in\rho_T(A_{p_k})\),
\[
x\wedge y
=
\operatorname{maj}(x,y,\rho_T(d)),
\qquad
x\vee y
=
\operatorname{maj}(x,y,\rho_T(d^{\operatorname{opp}})).
\]
Since $\rho_T(A_{p_k})$ is majority-closed, the preceding identities imply that $x\wedge y,x\vee y\in\rho_T(A_{p_k})$.
Hence $\rho_T(A_{p_k})$ is closed under componentwise meet and join
for every $k\in[m]$.

Finally, fix \(p\in\mathbb R^n\).
Every affinity domain is a face of some maximal affinity domain
\cite[Lemma~2.5(3)]{iwamasa2026towards}.
Hence \(\overline{A_p}\) is a face of
\(\overline{A_{p_k}}\) for some \(k\in[m]\), and therefore
\(\rho_T(\overline{A_p})\) is a face of
\(\rho_T(\overline{A_{p_k}})\).

Take \(x,y\in\rho_T(A_p)\).
Since \(\rho_T(A_{p_k})\) is closed under componentwise meet and join,
both \(x\wedge y\) and \(x\vee y\) belong to \(\rho_T(A_{p_k})\).
Moreover, since \(x,y\in\rho_T(\overline{A_p})\), their midpoint
belongs to \(\rho_T(\overline{A_p})\).
Since
\((x+y)/2=((x\wedge y)+(x\vee y))/2\)
and \(\rho_T(\overline{A_p})\) is a face of
\(\rho_T(\overline{A_{p_k}})\), we obtain
\(x\wedge y,x\vee y\in\rho_T(\overline{A_p})\).
Finally,
\(x\wedge y,x\vee y\in\{0,1\}^n\) and
\(\rho_T(\overline{A_p})\cap\{0,1\}^n=\rho_T(A_p)\).
Therefore \(x\wedge y,x\vee y\in\rho_T(A_p)\), and
\(\rho_T(A_p)\) is closed under componentwise meet and join.
\qed
\end{pf}

We can now prove the characterization for functions on a unit hypercube.

\begin{theorem}[Functions on the vertices of a unit hypercube]
\label{thm:functions-on-unit-hypercube}
Let \(f:\{0,1\}^n\to\mathbb R\cup\{+\infty\}\) be a function, and let \(\tilde f\) be its local convex extension on \([0,1]^n\).
Then the following conditions are equivalent.
\begin{enumerate}
\item[\textup{(i)}]
The function \(f\) is \(\mathrm{L}\)-extendable through its local convex extension, that is, the restriction of \(\tilde f\) to \(\{0,\frac12,1\}^n\) is \(\frac12\mathrm{UJ}\)-convex.
\item[\textup{(ii)}]
The function \(f\) is \(\mathrm{D}\)-convex.
\end{enumerate}
\end{theorem}

\begin{pf}
We first prove \(\textup{(i)}\Rightarrow\textup{(ii)}\). 
For \(p\in\mathbb R^n\), put 
\( 
A_p := \operatorname*{argmin}_{x\in\{0,1\}^n} f[-p](x). 
\) By Lemma~\ref{lem:common_reflection_for_all_argmins}, there exists \(T\subseteq[n]\) such that, for every \(p\in\mathbb R^n\), the set \(\rho_T(A_p)\) is closed under componentwise meet and join. 
Define \(g:=f\circ\rho_T\).
By \eqref{eq:reflected-meet-join} and
\eqref{eq:d-convex-function-unit-hypercube}, it suffices to show that
\(g\) is submodular.
Coordinate reflection transforms a linear perturbation of \(g\) into
a linear perturbation of \(f\), up to an additive constant.
Hence, as the perturbation vector ranges over \(\mathbb R^n\), the minimizer sets of the linear perturbations of \(g\) are exactly the sets \(\rho_T(A_p)\). Since all these sets are closed under componentwise meet and join, the characterization of submodularity by minimizer sets of linear perturbations~\cite[Theorem~7.17]{murota2003discrete} implies that \(g\) is submodular. Therefore \(f\) is \(\mathrm D\)-convex.

We next prove
\(\textup{(ii)}\Rightarrow\textup{(i)}\).
Assume that \(f\) is \(\mathrm{D}\)-convex.
Then there exists \(T\subseteq[n]\) such that
\(f\circ\rho_T\) is submodular.
Since the local convex extension of a submodular function on
\(\{0,1\}^n\) coincides with its Lovász extension,
\(\tilde f\) is convex and piecewise affine with respect
to the reflected Freudenthal simplicial division determined by \(T\).

By Proposition~\ref{prop:discordant-freudenthal-refinement}, this
reflected Freudenthal simplicial division is a coarsening of the
\(\frac12\mathrm{UJ}\) simplicial division.
Consequently, \(\tilde f\) is affine on every
\(\frac12\mathrm{UJ}\)-simplex.
Thus, \(\tilde f\) is the simplicial extension, with respect to the
\(\frac12\mathrm{UJ}\) simplicial division, of its restriction to
\(\{0,\frac12,1\}^n\).
Since \(\tilde f\) is convex,
Proposition~\ref{prop:half-uj-convex-extension-characterization}
implies that this restriction is
\(\frac12\mathrm{UJ}\)-convex.
This proves \(\textup{(ii)}\Rightarrow\textup{(i)}\).\qed
\end{pf}

\subsection{Integer-lattice characterization}
\label{subsec:integer-lattice-characterization}
We now extend the characterization of
Theorem~\ref{thm:functions-on-unit-hypercube}
from a single unit hypercube to the integer lattice.
In addition to $\mathrm{D}$-convexity of each unit hypercube
restriction, integral convexity is needed to ensure
the global convexity of the local convex extension.
This yields the following characterization.

\begin{theorem}[Functions on the integer lattice]
\label{thm:local-l-extendability-on-lattice}
Let \(f:\mathbb Z^n\to\mathbb R\cup\{+\infty\}\), and let
\(\tilde f\) be its local convex extension.
For \(z\in\mathbb Z^n\), define
\(f_z:\{0,1\}^n\to\mathbb R\cup\{+\infty\}\) by
\(f_z(u):=f(z+u)\).
Then the following conditions are equivalent.
\begin{enumerate}
\item[\textup{(i)}]
The function \(f\) is \(\mathrm{L}\)-extendable through its local convex extension \(\tilde f\); that is,
\(\tilde f|_{(\frac12\mathbb Z)^n}\) is \(\frac12\mathrm{UJ}\)-convex.

\item[\textup{(ii)}]
The function \(f\) is \(\mathrm{L}\)-extendable, and every
\(f_z\) is \(\mathrm{D}\)-convex.

\item[\textup{(iii)}]
The function \(f\) is integrally convex, and every
\(f_z\) is \(\mathrm{D}\)-convex.
\end{enumerate}
\end{theorem}

\begin{pf}
We first note that, for every \(z\in\mathbb Z^n\), the local convex
extension of \(f_z\) satisfies
\begin{equation}
\tilde{f_z}(x)=\tilde f(z+x)
\qquad
(x\in[0,1]^n).
\label{eq:local-extension-commutes-with-cube-restriction}
\end{equation}
Indeed, this follows directly from the definition of the local convex
extension as the convexification on each integral neighborhood
\cite[Section~3.2]{murota2023recent}, since
\(N(z+x)=z+(N(x)\cap\{0,1\}^n)\).

We first prove \(\textup{(i)}\Rightarrow\textup{(ii)}\).
Under \(\textup{(i)}\), the local convex extension itself gives an
\(\mathrm{L}\)-convex relaxation of \(f\).
Moreover, for every \(z\in\mathbb Z^n\),
\eqref{eq:local-extension-commutes-with-cube-restriction} shows that
the restriction of \(\tilde{f_z}\) to
\(\{0,\frac12,1\}^n\) is \(\frac12\mathrm{UJ}\)-convex.
Hence Theorem~\ref{thm:functions-on-unit-hypercube} implies that
\(f_z\) is \(\mathrm{D}\)-convex.

The implication
\(\textup{(ii)}\Rightarrow\textup{(iii)}\)
follows from
Theorem~\ref{thm:l-extendability-implies-integral-convexity}.

Finally, we prove
\(\textup{(iii)}\Rightarrow\textup{(i)}\).
For every \(z\in\mathbb Z^n\), the \(\mathrm{D}\)-convexity of
\(f_z\), together with
Theorem~\ref{thm:functions-on-unit-hypercube}, shows that
\(\tilde{f_z}\) is the simplicial extension of its restriction to
\(\{0,\frac12,1\}^n\) associated with the
\(\frac12\mathrm{UJ}\) simplicial division.
By \eqref{eq:local-extension-commutes-with-cube-restriction}, these
local descriptions agree with \(\tilde f\) on every unit hypercube.
Since the unit hypercube divisions are the restrictions of the global
\(\frac12\mathrm{UJ}\) simplicial division, \(\tilde f\) is the
simplicial extension of its restriction to
\((\frac12\mathbb Z)^n\).
Since \(f\) is integrally convex, \(\tilde f\) is convex on
\(\mathbb R^n\).
Therefore, by the equivalence of conditions \(\textup{(i)}\) and
\(\textup{(ii)}\) in
Proposition~\ref{prop:half-uj-convex-extension-characterization},
\(\tilde f|_{(\frac12\mathbb Z)^n}\) is
\(\frac12\mathrm{UJ}\)-convex.
\qed
\end{pf}

Unlike $\mathrm{D}$-convexity on the integer lattice,
$\mathrm{D}$-convexity of every unit hypercube restriction
does not require compatibility of the reflected Freudenthal
structures across unit hypercubes.
The following example shows that the class characterized in
Theorem~\ref{thm:local-l-extendability-on-lattice}
strictly contains the class of $\mathrm{D}$-convex functions
on the integer lattice.

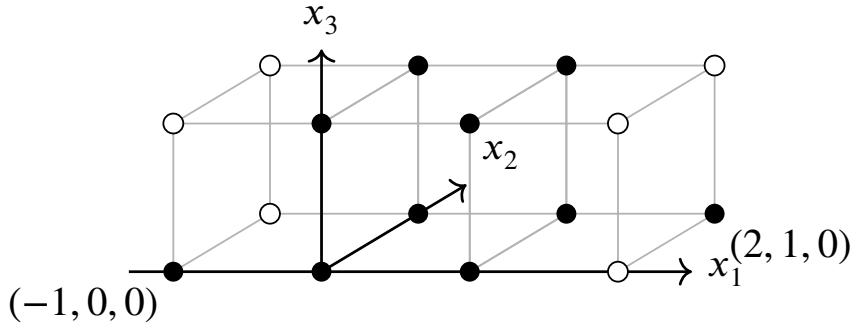
\begin{figure}
  \centering
  \resizebox{0.7\linewidth}{!}{%
    \begin{tikzpicture}[
        x={(1.15cm,0cm)},
        y={(0.75cm,0.45cm)},
        z={(0cm,1.15cm)}
      ]

      \foreach \a in {-1,0,1} {
        \draw[gray!60]
          (\a,0,0)--({\a+1},0,0)--({\a+1},1,0)--(\a,1,0)--cycle
          (\a,0,1)--({\a+1},0,1)--({\a+1},1,1)--(\a,1,1)--cycle
          (\a,0,0)--(\a,0,1)
          ({\a+1},0,0)--({\a+1},0,1)
          ({\a+1},1,0)--({\a+1},1,1)
          (\a,1,0)--(\a,1,1);
      }

      \draw[->,semithick] (-1.3,0,0) -- (2.5,0,0)
        node[right] {$x_1$};

      \draw[->,semithick] (0,0,0) -- (0,1.5,0)
        node[above right] {$x_2$};

      \draw[->,semithick] (0,0,0) -- (0,0,1.5)
        node[above] {$x_3$};

      \foreach \x in {0,1}
        \foreach \y in {0,1}
          \foreach \z in {0,1}
            \fill (\x,\y,\z) circle (2.2pt);

      \fill (-1,0,0) circle (2.2pt);
      \fill (2,1,0) circle (2.2pt);

      \foreach \x/\y/\z in {
        -1/1/0, -1/0/1, -1/1/1,
         2/0/0,  2/0/1,  2/1/1
      }
        \draw[fill=white] (\x,\y,\z) circle (2.2pt);

      \node[below left] at (-1,0,0) {$(-1,0,0)$};
      \node[below right] at (2,1,0) {$(2,1,0)$};

    \end{tikzpicture}%
  }
  \caption{
    Black points represent
    $X_3=\{0,1\}^3\cup\{(-1,0,0),(2,1,0)\}$;
    white points represent the other vertices of the adjacent cubes.
  }
  \label{fig:X3}
\end{figure}

\begin{example}[\(\mathrm{D}\)-convexity of unit hypercube restrictions
need not extend to the integer lattice]
\label{ex:local-d-convex-not-global-d-convex}
Let
\(X_3:=\{0,1\}^3\cup\{(-1,0,0),(2,1,0)\}\subseteq\mathbb Z^3\),
and let
\(\delta_{X_3}:\mathbb Z^3\to\mathbb R\cup\{+\infty\}\)
be the indicator function of \(X_3\).
Figure~\ref{fig:X3} illustrates the set \(X_3\).
The set \(X_3\) is integrally convex.
Indeed, a direct check shows that \(\overline{X_3}\) is the union of the convex hulls of the points of \(X_3\) contained in the three unit hypercubes \([-1,0]\times[0,1]^2\), \([0,1]^3\), and \([1,2]\times[0,1]^2\).

We next verify the \(\mathrm{D}\)-convexity of the translated
unit hypercube intersections.
The only choices of \(z\in\mathbb Z^3\) for which the translated intersection
\(
\bigl(X_3\cap(z+\{0,1\}^3)\bigr)-z
\)
is \(3\)-dimensional are
\(z^-=(-1,0,0)\), \(z^0=(0,0,0)\), and \(z^+=(1,0,0)\).
The corresponding translated intersections are
\[
Y^-:=\{(0,0,0)\}\cup\bigl(\{1\}\times\{0,1\}^2\bigr),
\qquad
\{0,1\}^3,
\qquad
Y^+:=\bigl(\{0\}\times\{0,1\}^2\bigr)\cup\{(1,1,0)\},
\]
respectively.
The sets \(Y^-\) and \(\{0,1\}^3\) are closed under the usual
componentwise meet and join, while
\(\rho_{\{3\}}(Y^+)\) is closed under componentwise meet and join.
Thus, all three \(3\)-dimensional restrictions are
\(\mathrm{D}\)-convex.
The remaining nonempty translated intersections are, up to coordinate
permutations and reflections,
\(\{\mathbf 0\}\), \(\{\mathbf 0,e_1\}\),
\(\{\mathbf 0,e_1,e_1+e_2\}\), or
\(\{\mathbf 0,e_1,e_2,e_1+e_2\}\).
Each of these sets is closed under componentwise meet and join, and
hence every remaining nonempty restriction is \(\mathrm{D}\)-convex.
Consequently, by
Theorem~\ref{thm:local-l-extendability-on-lattice}, the restriction of
the local convex extension of \(\delta_{X_3}\) to
\((\frac12\mathbb Z)^3\) is
\(\frac12\mathrm{UJ}\)-convex.

We now show that \(\delta_{X_3}\) is not \(\mathrm{D}\)-convex.
Suppose, to the contrary, that the reflected Freudenthal structures on
the unit hypercubes can be chosen compatibly.
First consider the left unit hypercube.
The set \(Y^-\) is \(3\)-dimensional, and its unique pair of opposite
vertices is
\(
(0,0,0), (1,1,1).
\)
If \(T^-\subseteq \{1, 2, 3\}\) is such that \(\rho_{T^-}(Y^-)\) is closed under componentwise meet and join, then the componentwise minimum and maximum of \(\rho_{T^-}(Y^-)\) are \(\mathbf 0\) and \(\mathbf 1\).
Their inverse images under \(\rho_{T^-}\) therefore form an opposite
pair in \(Y^-\).
By the uniqueness of this pair, we must have \(T^-=\emptyset\) or \(T^-=\{1, 2, 3\}\).
Both choices induce, on the face \(\{1\}\times[0,1]^2\) of the translated left cube, the diagonal joining \((1,0,0)\) and \((1,1,1)\).
Thus, in the original coordinates, the left cube requires the diagonal
joining
\(
(0,0,0)~\text{and}~(0,1,1)
\)
on the face \(\{0\}\times[0,1]^2\) of the central cube.
Next consider the right unit hypercube.
The set \(Y^+\) is \(3\)-dimensional, and its unique pair of opposite
vertices is
\(
(0,0,1), (1,1,0).
\)
The same argument shows that, if
\(\rho_{T^+}(Y^+)\) is closed under componentwise meet and join, then
\(T^+=\{3\}\) or \(T^+=\{1,2\}\).
Both choices induce, on the face
\(\{0\}\times[0,1]^2\) of the translated right cube, the diagonal
joining \((0,0,1)\) and \((0,1,0)\).
Thus, in the original coordinates, the right cube requires the diagonal
joining
\(
(1,0,1)~\text{and}~(1,1,0)
\)
on the face \(\{1\}\times[0,1]^2\) of the central cube.
The left and right cubes therefore require different diagonal types on the two opposite square faces of the central cube.
In a reflected Freudenthal triangulation of the central cube determined by a set \(T^0\subseteq \{1, 2, 3\}\), however, the diagonal type on each of these faces is determined by the same set \(T^0\cap\{2,3\}\).
More precisely, the two faces have one diagonal type when
\(\lvert T^0\cap\{2,3\}\rvert\) is even and the other diagonal type when it is odd.
Hence, the central cube cannot simultaneously realize the diagonal required by the left cube on \(\{0\}\times[0,1]^2\) and the diagonal required by the right cube on \(\{1\}\times[0,1]^2\).
Therefore, the local reflected Freudenthal structures cannot be chosen compatibly over \(\mathbb R^3\), and \(\delta_{X_3}\) is not \(\mathrm{D}\)-convex.
\end{example}

The inclusion relations among the function classes considered above
are summarized in Figure~\ref{fig:global-function-classes}.
All the inclusions shown in the figure are strict.

\begin{figure}
\centering
\includegraphics[width=.75\linewidth]{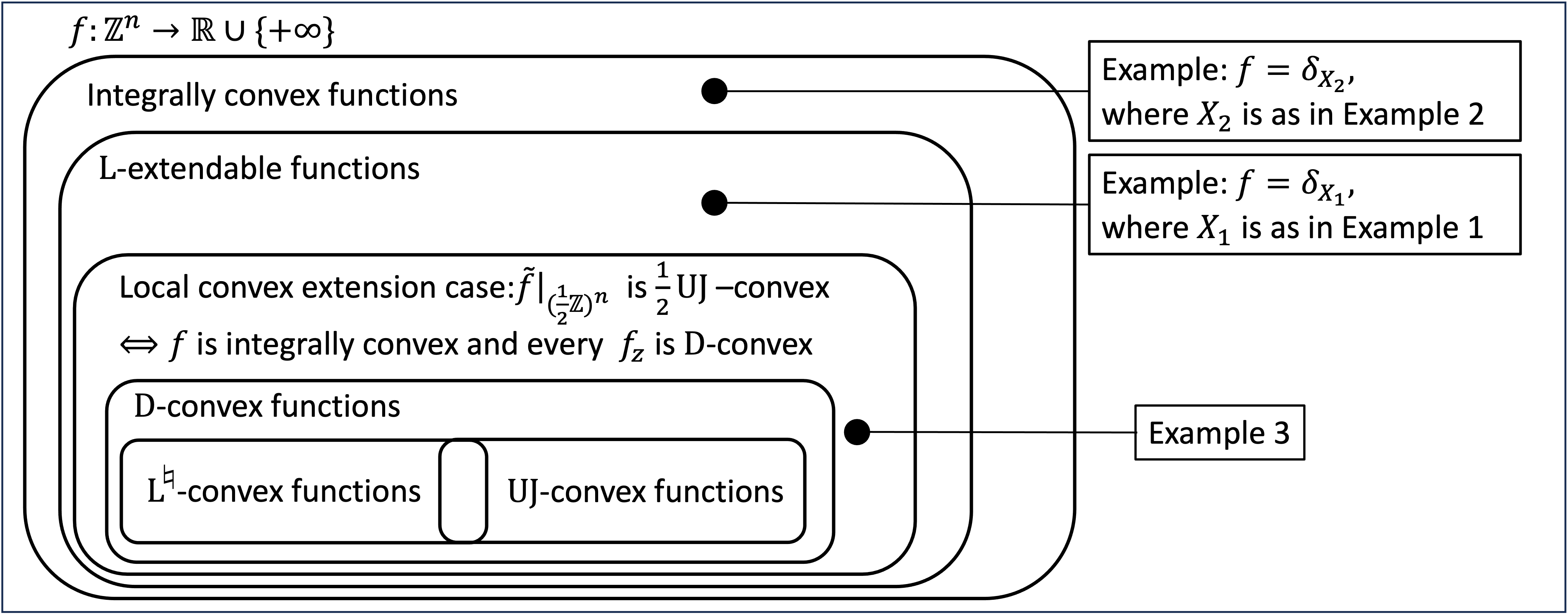}
\caption{
Relations among classes of functions on \(\mathbb Z^n\).
Every \(\mathrm{L}\)-extendable function is integrally convex by
Theorem~\ref{thm:l-extendability-implies-integral-convexity}, and the
inclusion is strict by
Example~\ref{ex:relative-opposite-not-majority}.
The class of functions that are \(\mathrm{L}\)-extendable through their
local convex extensions is characterized by
Theorem~\ref{thm:local-l-extendability-on-lattice}.
Example~\ref{ex:standard-simplex} shows that this class is
strictly smaller than the class of all \(\mathrm{L}\)-extendable
functions, while
Example~\ref{ex:local-d-convex-not-global-d-convex} shows that it is
strictly larger than the class of functions that are
\(\mathrm{D}\)-convex on the integer lattice.
}
\label{fig:global-function-classes}
\end{figure}


\section{Conclusion}
\label{sec:conclusion}

We showed that every $\mathrm{L}$-extendable function
on the integer lattice is integrally convex.
When the relaxation is prescribed to be the local convex
extension, we showed that, on a unit hypercube, this
property is equivalent to $\mathrm{D}$-convexity, and that,
on the integer lattice, it is equivalent to integral
convexity together with $\mathrm{D}$-convexity of every
unit hypercube restriction.
We also showed that this class strictly contains the
class of $\mathrm{D}$-convex functions on the integer lattice.
As a byproduct, we obtained a combinatorial characterization
of $\mathrm{D}$-convex subsets of the vertices of a unit
hypercube: such a nonempty set is $\mathrm{D}$-convex if
and only if it is majority-closed and contains a relative opposite pair.

Several questions remain open.
The main one is to characterize the integrally convex functions that
are $\mathrm{L}$-extendable.
It is also of interest to characterize when the reflecting sets
witnessing $\mathrm{D}$-convexity of the unit hypercube restrictions
can be chosen compatibly so that they determine a discordant
Freudenthal simplicial division of $\mathbb R^n$, and to develop
algorithms for testing the conditions obtained in this paper.

\appendix
\numberwithin{equation}{section}

\section{Rounding Relations for
Theorem~\ref{thm:l-extendability-implies-integral-convexity}}
\label{app:rounding-relations-l-extendability}

\begin{applemma}
\label{lem:rounding-relations-for-l-extendability}
Let \(x,y\in\mathbb Z^n\), and put
\(
z:=\frac{x+y}{2},
~a:=\frac{x+z}{2},
~b:=\frac{z+y}{2}.
\)
Define
\[
a^-:=\left[ a\right]_{\frac12},\qquad
a^+:=\left[ a\right]_{1},\qquad
b^-:=\left[ b\right]_{\frac12},\qquad
b^+:=\left[ b\right]_{1};\qquad
r:=\frac{a^++b^-}{2},\qquad
s:=\frac{a^-+b^+}{2}.
\]
Then the following hold:
\[
\textup{(i)}~[ r]_{\frac12}
=[ s]_{\frac12}=z;
\qquad
\textup{(ii)}~[r]_{1},
[s]_{1}
\in\mathbb Z^n\cap N(z);
\qquad
\textup{(iii)}~[r]_{1}
+[s]_{1}=2z.
\]
\end{applemma}

\begin{pf}
We verify the assertions componentwise.  Fix \(i\in[n]\).  Since
\(x_i+y_i=2z_i\), we have
\[
a_i+b_i
=
\frac{x_i+2z_i+y_i}{2}
=
2z_i.
\]

Suppose first that \(z_i\in\mathbb Z\).  Then
\(a_i,b_i\in\frac12\mathbb Z\), so the
\(\frac12\mathrm{UJ}\)-roundings leave both points unchanged.  Hence
\(a_i^-=a_i^+=a_i\) and \(b_i^-=b_i^+=b_i\), and therefore
\(
r_i=s_i=(a_i+b_i)/2=z_i.
\)
It follows immediately that
\(
\left[ r_i\right]_{\frac12}
=
\left[ s_i\right]_{\frac12}
=
z_i,
\)
that their upper roundings both equal \(z_i\), and that these upper roundings sum to \(2z_i\).

Suppose next that
\(z_i\in\frac12\mathbb Z\setminus\mathbb Z\).  Since \(x_i,y_i\in\mathbb
Z\), we then have
\(
a_i,b_i\in
\frac14\mathbb Z\setminus\frac12\mathbb Z.
\)
For a point in
\(\frac14\mathbb Z\setminus\frac12\mathbb Z\), its lower
\(\frac12\mathrm{UJ}\)-rounding is half-integral and nonintegral, whereas
its upper \(\frac12\mathrm{UJ}\)-rounding is integral.  Thus
\(r_i,s_i\in\frac14\mathbb Z\setminus\frac12\mathbb Z\).

Since \(a_i+b_i=2z_i\), we have
\[
r_i-z_i
=
\frac12\bigl[(a_i^+-a_i)+(b_i^--b_i)\bigr].
\]
Each term in brackets has absolute value \(1/4\), and hence
\(|r_i-z_i|\le 1/4\).  Similarly,
\[
s_i-z_i
=
\frac12\bigl[(a_i^--a_i)+(b_i^+-b_i)\bigr],
\]
and therefore \(|s_i-z_i|\le 1/4\).  On the other hand,
\(r_i-z_i\) and \(s_i-z_i\) belong to
\(\frac14\mathbb Z\setminus\frac12\mathbb Z\).  Consequently,
\[
|r_i-z_i|
=
|s_i-z_i|
=
\frac14.
\]

The two \(\frac12\mathrm{UJ}\)-roundings of a quarter-integral point add
up to twice that point.  Hence
\(a_i^-+a_i^+=2a_i\) and \(b_i^-+b_i^+=2b_i\), and thus
\(
r_i+s_i
=
(a_i^-+a_i^++b_i^-+b_i^+)/2
=
a_i+b_i
=
2z_i.
\)
Together with the preceding distance equalities, this gives
\[
\{r_i,s_i\}
=
\left\{
z_i-\frac14,\,
z_i+\frac14
\right\}.
\]

By the definition of the \(\frac12\mathrm{UJ}\)-roundings,
\[
\left[ z_i-\frac14\right]_{\frac12}
=
\left[ z_i+\frac14\right]_{\frac12}
=
z_i,
\]
while the corresponding upper roundings are
\(z_i-\frac12\) and \(z_i+\frac12\).  Therefore
\[
\left[ r_i\right]_{\frac12}
=
\left[ s_i\right]_{\frac12}
=
z_i.
\]
Moreover,
\([ r_i]_{1}\) and
\([ s_i]_{1}\) are integers at distance
\(1/2\) from \(z_i\), and hence belong to the \(i\)-th coordinate of
\(N(z)\).  Their sum is \(2z_i\).
Since \(i\in[n]\) was arbitrary, assertions \textup{(i)}--\textup{(iii)} follow.\qed
\end{pf}

%
\section*{Declaration of generative AI and AI-assisted technologies in the manuscript preparation process}

During the preparation of this work, Ken Yokoyama used ChatGPT, developed
by OpenAI, to assist with the writing and revision of the manuscript and
with checking the presentation of mathematical arguments and proofs.
The author reviewed and edited all AI-assisted content, independently
verified the mathematical statements and proofs, and takes full responsibility for the content of the published article.

\section*{Acknowledgements}
The authors thank Satoru Fujishige for suggesting the research direction pursued in this paper.

The first and third authors, Yokoyama and Kimura, were partially supported by JST ERATO, Grant Number JPMJER2301.
The first author, Yokoyama, was also supported by JST BOOST, Japan, Grant Number JPMJBS2406.
The second author, Iwamasa, was supported by JSPS KAKENHI Grant Numbers JP22K17854, JP24K21315, and JP24K02901.

\bibliographystyle{plainnat}
\bibliography{reference.bib}

@article{fujishige2014bisubmodular,
  title={Bisubmodular polyhedra, simplicial divisions, and discrete convexity},
  author={Fujishige, Satoru},
  journal={Discrete Optimization},
  volume={12},
  pages={115--120},
  year={2014}
}

@book{murota2003discrete,
  title={Discrete Convex Analysis},
  author={Murota, Kazuo},
  year={2003},
  publisher={SIAM}
}

@article{hirai2018convexity,
  title={L-convexity on graph structures},
  author={Hirai, Hiroshi},
  journal={Journal of the Operations Research Society of Japan},
  volume={61},
  number={1},
  pages={71--109},
  year={2018}
}

@article{murota2023recent,
  title={Recent progress on integrally convex functions},
  author={Murota, Kazuo and Tamura, Akihisa},
  journal={Japan Journal of Industrial and Applied Mathematics},
  volume={40},
  number={3},
  pages={1445--1499},
  year={2023}
}

@article{fujishige2022discrete,
  author  = {Fujishige, Satoru and Tardella, Fabio},
  title = {Discrete {2}-convex functions},
  journal = {Mathematical Programming},
  volume  = {195},
  number  = {1--2},
  pages   = {831--854},
  year    = {2022}
}

@inproceedings{huber2012towards,
  author    = {Huber, Anna and Kolmogorov, Vladimir},
  title     = {Towards minimizing {$k$}-submodular functions},
  booktitle = {Proceedings of the 2nd International Symposium on
               Combinatorial Optimization (ISCO)},
  pages     = {451--462},
  year      = {2012}
}

@article{hirai2016k,
  title={On k-submodular relaxation},
  author={Hirai, Hiroshi and Iwamasa, Yuni},
  journal={SIAM Journal on Discrete Mathematics},
  volume={30},
  number={3},
  pages={1726--1736},
  year={2016}
}

@article{FavatiTardella1990,
  author  = {Favati, Paola and Tardella, Fabio},
  title   = {Convexity in nonlinear integer programming},
  journal = {Ricerca Operativa},
  volume  = {53},
  pages   = {3--44},
  year    = {1990}
}

@article{moriguchi2020discrete,
  author  = {Moriguchi, Satoko and Murota, Kazuo and Tamura, Akihisa and Tardella, Fabio},
  title = {Discrete midpoint convexity},
  journal = {Mathematics of Operations Research},
  volume  = {45},
  number  = {1},
  pages   = {99--128},
  year    = {2020}
}

@article{hirai2015extendable,
  author  = {Hirai, Hiroshi},
  title   = {\(\mathrm{L}\)-extendable functions and a proximity scaling
             algorithm for minimum cost multiflow problem},
  journal = {Discrete Optimization},
  volume  = {18},
  pages   = {1--37},
  year    = {2015}
}

@book{fujishige2005submodular,
  author    = {Fujishige, Satoru},
  title     = {Submodular Functions and Optimization},
  edition   = {2},
  publisher = {Elsevier},
  year      = {2005}
}

@article{IimuraMurotaTamura2005,
  author  = {Iimura, Takuya and Murota, Kazuo and Tamura, Akihisa},
  title   = {Discrete Fixed Point Theorem Reconsidered},
  journal = {Journal of Mathematical Economics},
  volume  = {41},
  number  = {8},
  pages   = {1030--1036},
  year    = {2005}
}

@article{huber2014skew,
  author  = {Anna Huber and Andrei Krokhin and Robert Powell},
  title   = {Skew Bisubmodularity and Valued {CSP}s},
  journal = {SIAM Journal on Computing},
  volume  = {43},
  number  = {3},
  pages   = {1064--1084},
  year    = {2014}
}

@inproceedings{kolmogorov2011submodularity,
  author    = {Vladimir Kolmogorov},
  title     = {Submodularity on a Tree: Unifying
               {$\mathrm{L}^\natural$}-Convex and Bisubmodular Functions},
  booktitle = {Proceedings of the 36th International Symposium on
               Mathematical Foundations of Computer Science (MFCS)},
  pages     = {400--411},
  year      = {2011}
}

@inproceedings{gridchyn2013potts,
  author    = {Igor Gridchyn and Vladimir Kolmogorov},
  title     = {Potts Model, Parametric Maxflow and
               {$k$}-Submodular Functions},
  booktitle = {Proceedings of the IEEE International Conference on
               Computer Vision (ICCV)},
  pages     = {2320--2327},
  year      = {2013}
}

@article{iwata2016half,
  author  = {Yoichi Iwata and Magnus Wahlstr{\"o}m and Yuichi Yoshida},
  title   = {Half-Integrality, {LP}-Branching, and {FPT} Algorithms},
  journal = {SIAM Journal on Computing},
  volume  = {45},
  number  = {4},
  pages   = {1377--1411},
  year    = {2016}
}

@article{kimura2024neighborhood,
  author  = {Kei Kimura and Kotaro Nakayama},
  title = {Neighborhood persistency of the linear optimization
         relaxation of integer linear optimization},
  journal = {Mathematical Programming},
  volume  = {214},
  pages   = {133--151},
  year    = {2025}
}

@article{iwamasa2026towards,
  author  = {Yuni Iwamasa},
  title   = {Towards Discrete Convex Analysis over Classical Root Systems},
  journal = {{arXiv} preprint {arXiv}:2608.16090},
  year    = {2026}
}

@article{hirai2016discrete,
  author  = {Hiroshi Hirai},
  title   = {Discrete Convexity and Polynomial Solvability in Minimum 0-Extension Problems},
  journal = {Mathematical Programming},
  volume  = {155},
  number  = {1--2},
  pages   = {1--55},
  year    = {2016}
}

@incollection{todd1977union,
  author    = {Todd, Michael J.},
  title     = {Union Jack Triangulations},
  booktitle = {Fixed Points: Algorithms and Applications},
  publisher = {Academic Press},
  pages     = {315--336},
  year      = {1977}
}

@article{freudenthal1942simplicial,
  author  = {Freudenthal, Hans},
  title   = {Simplizialzerlegungen von beschr{\"a}nkter Flachheit},
  journal = {Annals of Mathematics},
  volume  = {43},
  number  = {3},
  pages   = {580--582},
  year    = {1942}
}

\end{document}